\documentclass[10pt,oneside,english,reqno]{amsart}
\usepackage{fullpage}

\usepackage{amsmath,amsfonts,amssymb}

\usepackage{graphicx,caption,subfig}
\usepackage{color}

\usepackage[latin9]{inputenc}
\usepackage[T1]{fontenc}

\usepackage[normalem]{ulem}

\usepackage{enumitem}
\usepackage{cases}
\usepackage{esint}
\usepackage[unicode=true,pdfusetitle,
 bookmarks=true,bookmarksnumbered=false,bookmarksopen=true,bookmarksopenlevel=2,
 breaklinks=false,pdfborder={0 0 1},backref=false,colorlinks=true]
 {hyperref}
\usepackage{breakurl}
 
\usepackage{prettyref}

\newrefformat{def}{Definition \ref{#1}}
\newrefformat{thm}{Theorem \ref{#1}}
\newrefformat{theo}{Theorem \ref{#1}}
\newrefformat{prop}{Proposition \ref{#1}}
\newrefformat{lem}{Lemma \ref{#1}}
\newrefformat{cor}{Corollary \ref{#1}}
\newrefformat{sec}{Section \ref{#1}}
\newrefformat{subsec}{Subsection \ref{#1}}
\newrefformat{rem}{Remark \ref{#1}}
\newrefformat{ex}{Example \ref{#1}}

\numberwithin{equation}{section}
\numberwithin{figure}{section}
  \theoremstyle{plain}
  \newtheorem{thm}{\protect\theoremname}[section]
  \theoremstyle{definition}
  \newtheorem{defn}{\protect\definitionname}[section]
  \theoremstyle{plain}
  \newtheorem{prop}{\protect\propositionname}[section]
  \theoremstyle{plain}
  
 \ifx\proof\undefined\
   \newenvironment{proof}[1][\proofname]{\par
     \normalfont\topsep6\p@\@plus6\p@\relax
     \trivlist
     \itemindent\parindent
     \item[\hskip\labelsep
           \scshape
       #1]\ignorespaces
   }{%
     \endtrivlist\@endpefalse
   }
   \providecommand{\proofname}{Proof}
 \fi
  \theoremstyle{plain}

  \theoremstyle{plain}

\providecommand{\corollaryname}{Corollary}
\providecommand{\definitionname}{Definition}
\providecommand{\lemmaname}{Lemma}
\providecommand{\propositionname}{Proposition}
\providecommand{\theoremname}{Theorem}
\providecommand{\examplename}{Example}

\usepackage{subfig}

\newtheorem{rem}{Remark}[section]

\def\NN{{\mathbb N}}
\def\RR{{\mathbb R}}
\def\CC{{\mathbb C}}

\def\w{\mathbf{w}}

\def\fn{\mathfrak{n}}

\def\l{\lambda}

\def\O{{\cal O}}

\def\C{{\mathcal C}}
\def\argmin{\mathop{\hbox{\textrm{arg min}}}}

\def\eref#1{(\ref{#1})}
\def\disp{\displaystyle}

\def\donchitre#1#2{\vskip 6.5cm\noindent
\parbox[t]{1in}{\special{eps:#1.eps x=6.5cm y=5.5cm}}
\hbox to 7cm{}\parbox[t]{0.0cm}{\special{eps:#2.eps x=6.5cm y=5.5cm}}}

\def\tn{|\!|\!|}

\global\long\def\lam{\mathcal{W}}
\global\long\def\ee{\varepsilon} \global\long\def\O{\Omega}
\global\long\def\w{\omega} \global\long\def\expn{\gamma}
\global\long\def\expnn{\delta}
\global\long\def\errfn{\phi}

\title{Stable soft extrapolation of entire functions}

\author{Dmitry Batenkov}\address{Department of Mathematics, Massachussetts Institute of Technology, Cambridge, MA 02139}\email{batenkov@mit.edu}
\author{Laurent Demanet}\address{Department of Mathematics, Massachussetts Institute of Technology, Cambridge, MA 02139}\email{laurent@math.mit.edu}
\author{Hrushikesh N. Mhaskar}\address{Institute of Mathematical Sciences, Claremont Graduate University, Claremont, CA 91711. }\email{hrushikesh.mhaskar@cgu.edu}

\begin{document}
\maketitle

\begin{abstract}
 
  Soft extrapolation refers to the problem of recovering a function
  from its samples, multiplied by a fast-decaying window and perturbed
  by an additive noise, over an interval which is potentially larger
  than the essential support of the window.  To achieve stable
  recovery one must use some prior knowledge about the function class,
  and a core theoretical question is to provide bounds on the possible
  amount of extrapolation, depending on the sample perturbation level
  and the function prior.

  In this paper we consider soft extrapolation of entire functions of
  finite order and type (containing the class of bandlimited functions
  as a special case), multiplied by a super-exponentially decaying
  window (such as a Gaussian).  We consider a weighted least-squares
  polynomial approximation with judiciously chosen number of terms and
  a number of samples which scales linearly with the degree of
  approximation. It is shown that this simple procedure provides
  stable recovery with an extrapolation factor which scales
  logarithmically with the perturbation level and is inversely
  proportional to the characteristic lengthscale of the function.
  The pointwise extrapolation error exhibits a H\"{o}lder-type
  continuity with an exponent derived from weighted potential theory,
  which changes from 1 near the available samples, to 0 when the
  extrapolation distance reaches the characteristic smoothness length
  scale of the function. The algorithm is asymptotically minimax, in
  the sense that there is essentially no better algorithm yielding
  meaningfully lower error over the same smoothness class.

  When viewed in the dual domain, soft extrapolation of an entire
  function of order 1 and finite exponential type corresponds to the
  problem of (stable) simultaneous de-convolution and super-resolution
  for objects of small space/time extent. Our results then show that
  the amount of achievable super-resolution is inversely proportional
  to the object size, and therefore can be significant for small
  objects. These results can be considered as a first step towards
  analyzing the much more realistic ``multiband'' model of a sparse
  combination of compactly-supported ``approximate spikes'', which
  appears in applications such as synthetic aperture radar, seismic
  imaging and direction of arrival estimation, and for which only
  limited special cases are well-understood.

\end{abstract}

\section*{Acknowledgements}

The research of DB and  LD is supported in part by AFOSR grant FA9550-17-1-0316, NSF
grant DMS-1255203, and a grant from the MIT-Skolkovo initiative. The research of HNM  is supported in part by ARO Grant W911NF-15-1-0385, and  by the Office of the Director of National Intelligence (ODNI), Intelligence Advanced Research Projects Activity (IARPA), via 2018-18032000002. The
authors would also like to thank Alex Townsend and Matt Li for useful
discussions.

\section{Introduction}

\subsection{Background}
Consider a one-dimensional (in space/time) object $F$, and suppose it
is corrupted by a low-pass convolutional filter $K$ and
additive modeling/noise error $E$, resulting in the
output $G$:
\[
  G\left(x\right)=\int K\left(x-y\right)F\left(y\right)dy+E\left(x\right).
\]
In the Fourier domain, we have
\begin{equation}\label{eq:fourier-tr-def}
\hat{F}(\w):=\int_\RR F(x)\exp(-i\w x)dx
\end{equation}
and consequently
\begin{equation}
  \hat{G}\left(\omega\right)=\hat{K}\left(\omega\right)\hat{F}\left(\omega\right)+\hat{E}\left(\omega\right),\label{eq:freq-samples}
\end{equation}
where $\hat{K}\left(\w\right)$ has small frequency support $\O_*$,
the so-called ``effective bandlimit'' of the system (for example as in
the case of the ideal low-pass filter
$\hat{K}=\chi_{\left[-\O_*,\O_*\right]}$). The problem of
computational super-resolution asks to recover features of
$F\left(x\right)$ from $G\left(x\right)$ below the classical
Rayleigh/Nyquist limit $\frac{\pi}{\O_*}$
\cite{bertero1996iiisuperresolution,lindberg2012mathematical}. As an
ill-posed inverse problem, super-resolution can be regularized using
some a-priori information about $F$. One of the main theoretical
questions of interest is to quantify the resulting stability of
recovery.

Viewed directly in the frequency domain, super-resolution is
equivalent to \emph{out-of-band extrapolation} of
$\hat{F}\left(\omega\right)$ for $\left|\omega\right|>\O_*$ from
samples of $\hat{G}\left(\omega\right)$. Since the Fourier transform
is an analytic function, this leads to the  the problem of stable
\emph{analytic continuation}, widely studied during the last couple of
centuries (\prettyref{subsec:novelty-related-work}).

\subsection{Contributions}

In this paper we consider the question of stable extrapolation of
certain class of entire functions $\hat{F}$. In more detail, we assume
that $\hat{F}$ is analytic in $\CC$ and, for some $\tau>0$,
\begin{equation}\label{eq:exp-type-class}
  \sup_{|z|>0,\;z\in\CC} \left| \hat{F}\left(z\right)\right| \exp\left(-\tau\left|z\right|\right)\leq 1.
\end{equation}

The underlying motivation for choosing such a growth condition is that
the resulting set contains Fourier
transforms of distributions of compact support (see
\prettyref{subsec:discussion} below for more details). Another common
name for such $\hat{F}$ would be ``bandlimited'' (or, in this case, ``time-limited'').

Furthermore, instead of the ``hard'' cutoff $\hat{K}=\chi_{\left[-\O_*,\O_*\right]}$ we
consider ``soft'' windows of super-exponentially decaying shapes,
parametrized by $\alpha\ge2$  (see \prettyref{fig:window})
\begin{equation}\label{eq:window-shape}
  \hat{K}\left(\omega\right)=w_{\alpha}\left(\omega\right):=\exp\left(-\left|\omega\right|^{\alpha}\right).
\end{equation}

Gaussian point-spread functions are considered a fairly reasonable
approximation in microscopy \cite{stallinga_accuracy_2010}, and when
$\alpha$ is increased the shape of $\hat{K}$ approaches the ideal
filter. We therefore  argue that the assumption
\eqref{eq:window-shape} is realistic in applications.

We further assume that the perturbation $\hat{E}\left(\omega\right)$
in \eqref{eq:freq-samples} is a uniformly bounded function
\begin{equation}
  \left|\hat{E}\left(\omega\right)\right|\leq\ee,\;\omega\in\mathbb{R}.\label{eq:uniform-error-bound}
\end{equation}

The ``soft extrapolation'' question, schematically depicted
in \prettyref{fig:extrapolation-schematic} is then to recover $\hat{F}\left(\omega\right)$
in a stable fashion, over an interval $|\w|\leq \O'$ which is potentially larger
than the effective support of the window $|\w|\leq\O_*$, where both
$\O'$ and $\O_*$ may depend on $\ee$ and $\alpha,\tau$.

\begin{figure}
  \begingroup \captionsetup[subfigure]{width=0.65\linewidth}
  \subfloat[Schematic spectrum extrapolation from noisy and
  bandlimited data. The original function is $\hat{F}$ (green,
  dashed). After being multiplied by the window $\hat{K}=w_{\alpha}$
  (brown, dash-dot), it gives the blue, solid curve. Adding a
  corruption $\hat{E}$ of size $\ee$, this gives the data $\hat{G}$
  (shaded grey). The question is to recover $\hat{F}(\omega)$ over
  as wide an interval as possible, with as good an accuracy as
  possible.]{\centering\includegraphics[width=0.7\columnwidth]{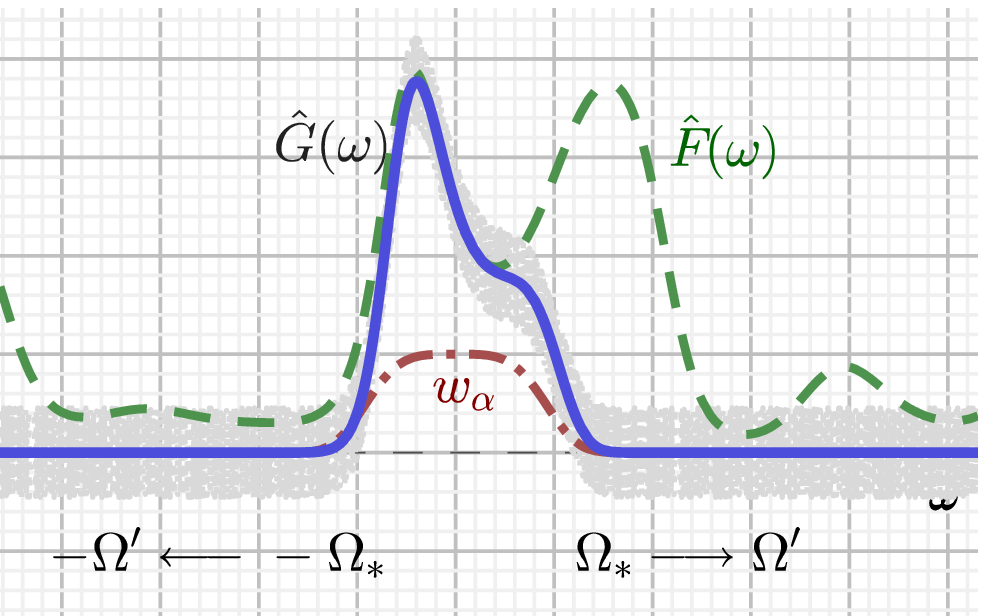}

    \label{fig:extrapolation-schematic}}

  \endgroup

  \begingroup \captionsetup[subfigure]{width=0.4\columnwidth}
  \subfloat[The object prior: the function $F$ has small space/time
  extent.]{\centering\includegraphics[width=0.4\columnwidth]{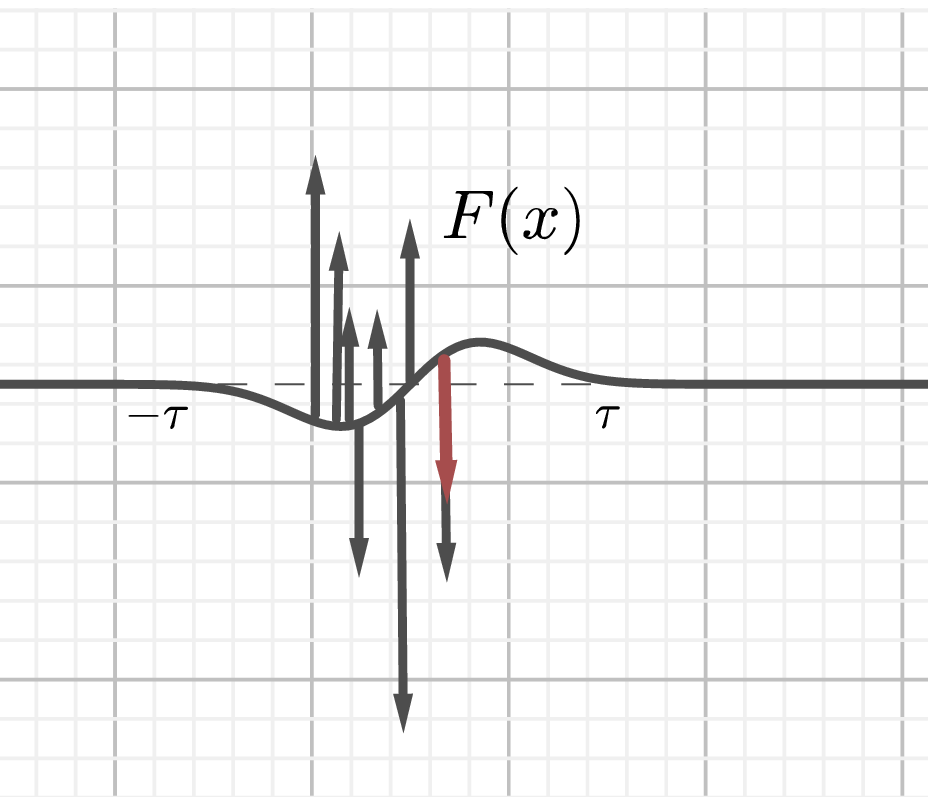}

    \label{fig:prior}} \subfloat[The super-exponentially decaying
  windows $\hat{K_{\alpha}}(\omega)=w_{\alpha}(\omega)$ (left column)
  and their inverse Fourier transforms (time domain filters
  $K_{\alpha}(x)$).]{\centering\includegraphics[width=0.5\columnwidth]{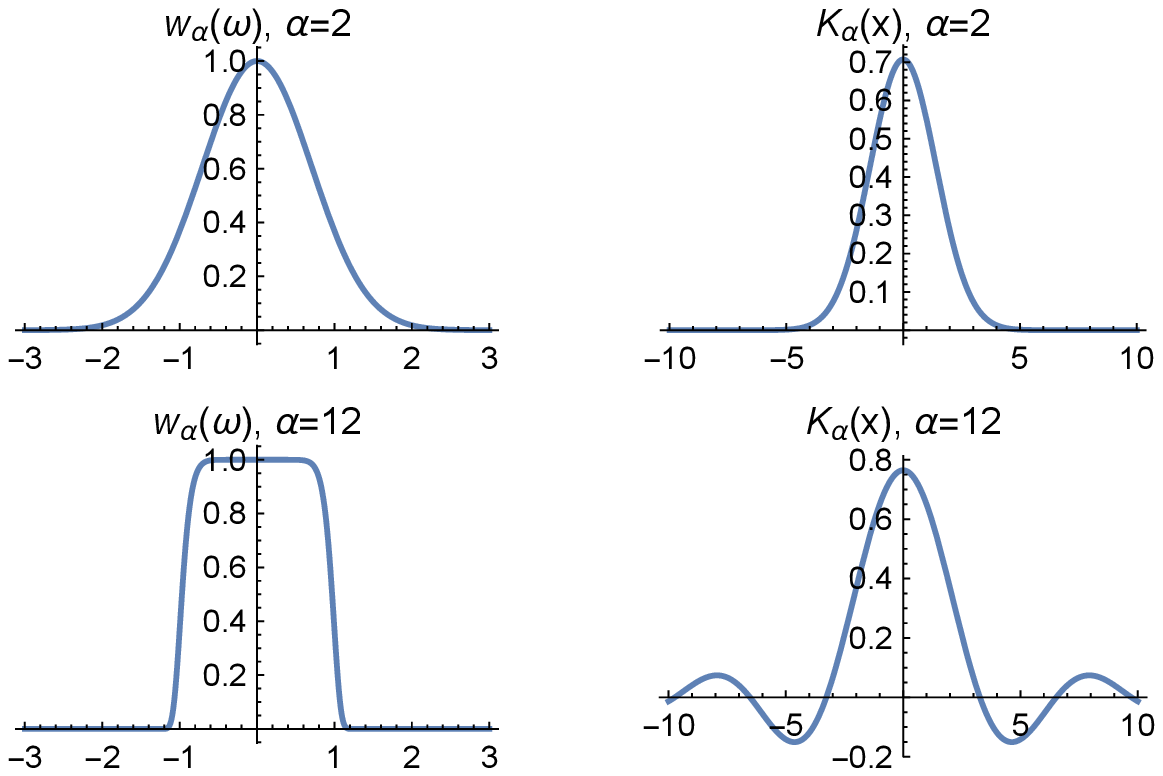}

    \label{fig:window}} \endgroup

  \caption{The schematic representation of the ``soft extrapolation''
    problem.}

\end{figure}

Our first main result, \prettyref{thm:main-thm-exp-type} below (in
particular see also \prettyref{rem:asymptotic-orders}), shows that a
weighted least-squares polynomial approximation with sufficiently
dense samples of the form (\ref{eq:freq-samples}), taken inside the
interval $\left[-\O_*,\O_*\right]$ with $\O_*$ scaling (up to
sub-logarithmic factors\footnote{More complete scaling w.r.t $\ee$, made precise
  in the sequel, is $\O_* \sim \left({1\over\log\log{1\over\ee}}\log{1\over\ee}\right)^{1\over\alpha}$}) like
$\left(\log {1\over\ee}\right)^{1\over\alpha}$ and a judiciously
chosen number of terms achieves an extrapolation factor
$\O'\over \O_*$ which scales (again, up to sub-logarithmic in
$1\over\ee$ factors) like
${1\over\tau}\bigl(\log{1\over\ee}\bigr)^{1-{1\over\alpha}}$, while
the pointwise extrapolation error exhibits a H\"{o}lder-type
continuity, morally of the form $\ee^{\expn(\w)}$ for $|\w|<\O'$. The
exponent $\expn$ varies from $\expn(0)=1$ to $\expn(\O')=0$, has an
explicit form originating from weighted potential theory, and has
itself a minor dependence on $\ee$ that we make explicit in the
sequel.

Our second main result, \prettyref{thm:lower-bound-eps} below, shows
that the above result is minimax in terms of optimal recovery and thus
cannot be meaningfully improved with respect to the asymptotic
behaviour in $\ee\ll 1$.

In fact, we prove more general estimates, which hold for functions
$\hat{F}$ of finite exponential order $\lambda\ge1$ and
type $\tau>0$, satisfying
\begin{equation}
  \label{eq:our-functional-class}
  \sup_{|z|>0}\left|\hat{F}(z)\right|\exp\left(-\tau|z|^\lambda\right)\leq 1
\end{equation}
provided that the window parameter satisfies $\alpha>\lambda$.

\subsection{Novelty and related work}
\label{subsec:novelty-related-work}

Analytic continuation is a very old subject since at least the times
of Weierstraß \cite{henrici_algorithm_1966}. It is the classical
example of an ill-posed inverse problem, and has been considered in
this framework since the early 1960's
\cite{miller_three_1964,miller_least_1970,cannon_problems_1965,miller1973onthe},
\cite{bertero1979onthe,bertero1980resolution,landau1986extrapolating},
and more recently \cite{demanet2016stableextrapolation} and
\cite{trefethen_quantifying_}.  These works establish regularized
linear least squares as a near-optimal computational method, while
also deriving the general form of Hölder-type stability, logarithmic
scaling of the extrapolation range and the connection to potential
theory. Weighted approximation of entire functions has been previously
considered in \cite{canadapap, mhaskar2002onthe}. Below we outline
some of these results in more detail, and relate them to ours.

\begin{enumerate}
\item Miller~\cite{miller_least_1970} and
  Miller\&Viano~\cite{miller1973onthe} considered the problem of
  analytic continuation in the general framework of ill-posed inverse
  problems with a prescribed bound. A general stabiliy estimate of the
  form $\ee^{w(z)}$ was established using Carleman's inequality, and
  regularized least squares (also known as the Tikhonov-Miller
  regularization) was shown to be a near-optimal recovery method. The
  setting in these work is one of ``hard'' extrapolation, where $F$ is
  analytic in an open disc $D$, is being measured on a compact closed
  curve $\Gamma\subset D$, and extended into $D$. It is assumed that
  $F$ is uniformly bounded by some constant on $\partial D$.

  In contrast, we consider analytic continuation of entire functions
  satisfying \eqref{eq:our-functional-class} on unbounded domains from
  samples on intervals of the form $\Gamma=\left[-\O_*,\O_*\right]$
  into $D=\{|z|\leq \O'\}$ where both $\O_*$ and $\O'$ can be
  arbitrarily large when  $\ee\ll 1$. It doesn't
  appear obvious how to extend the methods in
  \cite{miller_least_1970,miller1973onthe} to this setting, without
  using tools of weighted approximation of entire functions which were
  developed later.

\item Tikhonov-Miller theory was applied in the particular context of
  super-resolution in ~\cite{bertero1979onthe}. The inverse problem of
  ``optical image extrapolation'', i.e. restoring a square-integrable
  bandlimited function from its measurments on a bounded interval is
  precisely dual to extrapolating the Fourier transform of a
  space-limited object of finite energy. For the ``hard window''
  $\chi_{\left[-1,1\right]}$ and any bandlimit $c>0$ the authors showed
  that the best possible reconstruction error has a Hölder-type
  scaling $\ee^\alpha$ for some $0<\alpha<1/2$, in the uniform norm.

  In contrast, we consider the ``soft'' formulation, where the
  convolutional kernel $\hat{K}$ does not vanish, but rather becomes
  exponentially small at high frequencies. Our estimates give
  \emph{point-wise} error bounds so that the Hölder exponent is a
  function of the particular point, while also providing the
  functional dependence on the bandlimit $c$ (corresponding to $\tau$
  in our notations).

\item Out-of-band extrapolation for bandlimited functions was also
  considered in \cite{landau1986extrapolating}. This setting is
  closely related to ours as the sampling interval can be
  arbitrary. It was shown that in this case, the optimal extrapolation
  length scales logarithmically with the signal-to-noise ratio,
  however no pointwise estimates were obtained. The sampling widow was
  again the ``hard'' one $\hat{K}=\chi_{\left[-\O,\O\right]}$.

  Our results provide a more complete extrapolation scaling,
  $\tau^{-1}\log{1\over\ee}$ rather than just Landau's
  $\log{1\over\ee}$, while also giving pointwise error estimates. We
  also allow for functions of exponential order $\lambda\ge1$, and
  taking $\alpha\to\infty$ and $\lambda=1$, we in fact recover
  Landau's scalings.

\item A recent work by Demanet\&Townsend
  \cite{demanet2016stableextrapolation} investigates the question of
  stable extrapolation of analytic functions having convergent
  Chebyshev polynomial expansions in a Bernstein ellipse with
  parameter $\rho$ from equispaced data on $\left[-1,1\right]$. Linear
  least squares reconstruction is shown to be asymptotically optimal,
  if the number of samples scales quadratically with $M$, the number
  of terms in the approximation, and also $M\sim \log{1\over\ee}$. The
  resulting extrapolation error scales as $\ee^{\alpha(x)}$ where
  $\alpha$ has a simple dependence on $\rho$. Naturally, the maximal
  extrapolation extent is the boundary of the ellipse,
  i.e. $x_{max}={(\rho+\rho^{-1})\over 2}$, and also $\alpha(1)=1$,
  while $\alpha(x_{max})=0$.

  Our results are close in spirit to
  \cite{demanet2016stableextrapolation}, in particular the scaling
  $\fn\approx \log{1\over\ee}$ for the polynomial approximation order,
  however the specific setting of soft extrapolation is
  different. Furthermore, we do not require equispaced samples (but
  only bound the sample density), and
  also we obtain point-wise extrapolation error bounds on the entire
  complex plane rather than just the real line.

\item In another related recent work \cite{trefethen_quantifying_},
  Trefethen considers stable analytic continuation from an interval
  $E$ to an infinite half-strip (the so-called ``linear'' geometry),
  and shows an exponential loss of significant digits as function of
  distance from $E$ (corresponding to a stability estimate of the form
  $\ee^{\alpha(x)}$ where $x$ is the distance from $E$ and $\alpha$
  decreases exponentially). In contrast, we consider functions
  analytic in the entire plane $\CC$ which is not conformally
  equivalent to the linear geometry.
  
\item Extrapolation is closely related to approximation, and in our
  case of ``soft window'', one is naturally lead to adding a
  corresponding weight function. Indeed, the theory of weighted
  approximation (in particular of entire functions, \cite{canadapap,
    mhaskar1997introduction,mhaskar2002onthe}) plays a crucial role in
  our developments. In addition to extending these works to the
  extrapolation setting, our numerical approximation scheme in this
  paper does not necessarily require construction of quadrature
  formulas and orthogonal polynomials with respect to the weight
  $w_\alpha^2$.
\end{enumerate}
\subsection{Organization of the paper}
In \prettyref{sec:Main-results} we define the basic notation and
formulate the main theorems.  In \prettyref{sec:numsect} we present
some numerical experiments which demonstrate the results in
the case of $\alpha=2$ and $\lambda=1$. In \prettyref{sec:backsect} we
quote results from weighted approximation theory which are
subsequently used in \prettyref{sec:proofs} for the proofs.

\subsection{Discussion}
\label{subsec:discussion}

When viewed in the dual domain (the $x$ variable in
\eqref{eq:fourier-tr-def}), soft extrapolation of an entire function of
order $\lambda=1$ and finite exponential type $\tau>0$ corresponds to
the problem of (stable) simultaneous de-convolution and
super-resolution for objects $F$ of small space/time extent
$\tau$. Indeed, if $F\in L^1$, $\|F\|_1\le 1$, $F$ is even, and
$F(t)=0$ if $|t|>\tau$, then the Fourier transform $\hat{F}$ in
\eqref{eq:fourier-tr-def} is an entire function satisfying
\eqref{eq:exp-type-class} (another common name for $\hat{F}$ would be
``bandlimited'', although in this case it would be more appropriate to
say ``time-limited''). This is an example of theorems of Paley-Wiener
type (\cite[Sect. 7.2]{strichartz_guide_2003}, see also
\cite[p.12]{paley1934fourier}), which hold also for more general
classes such as tempered distributions.

In various applications such as seismic imaging,
communications, radar, and microscopy, a fairly realistic prior on $F$
takes the form of a sparse atomic combination of compactly supported
waveforms, also known as the multiband model in the literature
\cite{zhu_approximating_2017,mishali_theory_2010,izu_time-frequency_2009,venkataramani_optimal_2001}
\begin{equation}\label{eq:multiband}
  F\left(x\right)\sim\sum_{j=1}^{R}F_{j}\left(x-x_{j}\right),
\end{equation}
where each $F_{j}\left(x\right)$ is assumed to have a small space/time
support but is otherwise unknown. While there exist numerous studies
of multiband signals such as the ones quoted above, super-resolution
properties associated with this model are not well-understood, except
in only some special cases (see e.g.
\cite{akinshin_accuracy_2015,batenkov2016stability,batenkov_accuracy_2013,batenkov_geometry_2014,candes2013super,cand`es2014towards,demanet2014therecoverability,donoho1992superresolution,singdet,li_stable_2017,trigwave,loctrigwave,prony_essai_1795,stoica2005spectral}
as a very small sample).

Our results in this paper may be interpreted in this context when
instead of the sparse sum in \eqref{eq:multiband} we can consider the
limit of a single object (possibly a distribution) $F$ of compact
space/time support $\tau>0$ (\prettyref{fig:prior}). In particular, we
obtain the best possible scalings for stability of this inverse
problem, showing that the amount of achievable super-resolution scales
like ${1\over\tau}\log{1\over\ee}$, and therefore can be significant
for small objects with $\tau\ll 1$. Furthermore, a simple algorithm
--- linear least squares fitting --- is asymptotically optimal.

\section{Optimal extrapolation of entire functions}\label{sec:Main-results}

\subsection{Notation}
In the sequel, we fix $\alpha\ge 2$, and omit its mention from
notations except to avoid conflict of notation, and for
emphasis. Also, contrary to the introductory section, in the remainder
of the paper we denote the extrapolation variable by $x$ instead of
$\w$. The functions $\hat{F}$, $\hat{G}$, $\hat{E}$ will correspond to
$f,g$ and $\errfn$.

We shall use the standard definitions and notations of the spaces
$L^p$ for $0<p\leq\infty$ (in this paper with respect to the Lebesgue
measure on $\RR$) and the corresponding norms $\|\cdot\|_p$.

Given a function $f:\CC\to\CC$, and $\tau>0,\;\lambda\ge 1$ we
define
\begin{equation}
  \label{eq:tn-def}
  \tn f\tn_{\tau,\lambda}:=\sup_{|z|>0}|f(z)|\exp(-\tau|z|^\lambda).
\end{equation}

 \begin{defn}\label{def:entire_class_def}
 Given $\tau>0,\;\lambda\ge 1$, the class $B_{\tau,\lambda}$ consists of all entire
 functions $f$, real valued on $\RR$, and satisfying the condition 
 \begin{equation}\label{eq:entire_class_def}
 \tn f \tn_{\tau,\lambda}\le 1.
 \end{equation}
 \end{defn}

 \begin{rem}
   Without loss of generality, in this paper we restrict the
   considerations to the class $B_{\tau,\lambda}$, although it is a
   proper subset of the set of entire functions of exponential order
   $\lambda$ and type $\tau$\footnote{The standard definition of order
     and type is as follows (see e.g. \cite{levin_lectures_1996}): if
     \mbox{$M(f,r):=\sup_{0<|z|\leq r} |f(z)|$}, then
     \mbox{$\lambda:=\limsup_{r\to\infty} {\log\log M \over \log r}$}
     and \mbox{$\tau:=\limsup_{r\to\infty} {\log M\over r^\lambda}$.}}.

   Indeed, for fixed $\lambda, \tau$, any $f$ (real-valued on $\RR$)
   of order $\lambda$ and type $\tau$ and any $\tau'>\tau$ we have
   $\tn f \tn_{\tau',\lambda} < \infty$, and therefore $f=cf_0$ for
   some constant $c$ and $f_0\in B_{\tau',\lambda}$. Furthermore, if
   the original $f$ is complex-valued on $\RR$, one can consider the
   approximation of its real and imaginary parts separately, without
   changing the main asymptotic behaviour of the bounds.
 \end{rem}

 \begin{defn}
   \label{def:asymptotics-eps}

   When comparing small functions of a small quantity $\ee\to 0$, we
   shall write $a\left(\ee\right)\lessapprox b\left(\ee\right)$ when
   there exists $\ee_0$, depending on $\alpha,\tau,\lambda$ only, such that
   for all $c_1 > 0$ (however small), there exists $c_2 > 0$,
   depending on $\alpha,\tau,\lambda,c_1$ such that
   $$
   a(\ee) \leq c_2 \left({1\over\ee}\right)^{c_1/
     {\log\log{1\over\ee}}} b(\ee),\qquad \ee<\ee_0.
   $$
   When both $a\lessapprox b$ and $b\lessapprox a$ we shall sometimes
   write $a\approx b$.

 \end{defn}

% \begin{rem}
%   Essentially $a(\ee)\lessapprox b(\ee)$ means that\footnote{In fact,
%     we could have made the following definition: $a\lessapprox b$ when
%     for any $\nu>0$ there exists $c_\nu$ and $\ee_0$ such that for all
%     $\ee<\ee_0$ we have
%     $a(\ee)\leq c_{\nu} \ee^{-\nu/\log\log{1\over\ee}} b(\ee) $.}
%   $a(\ee)\leq \ee^{-o(1)}b(\ee)$.
% \end{rem}

  For $y>0$, $\Pi_y$ denotes the class of all algebraic polynomials of degree at
 most $y$. 
 This is the same as the class of polynomials of degree at most $\lfloor y\rfloor$, but the notation is simplified if we simply interpret $\Pi_y$ in this way, rather than writing $\Pi_{\lfloor y\rfloor}$.

\subsection{Extrapolation by least squares fitting}

 Let $\C=\{x_{M}<\cdots <x_{1}\}$ be a set of arbitrary real numbers. We observe data of the form
 \begin{equation}\label{eq:data_def}
 g(x)=w_\alpha(x)f(x)+\errfn(x), \qquad x\in\C,
 \end{equation}
 where $f \in B_{\tau,\lambda}$ and $\errfn$ is a function satisfying 
 \begin{equation}\label{eq:epsbd}
 |\errfn(x)|\le \ee, \qquad x\in\C.
 \end{equation}

 Given $\C$ and $g$ as above, we define the operator $S_n$ computing
 the solution to the following least squares problem of degree $n$:
 \begin{equation}\label{eq:lsquare_def}
 \begin{split}
 S_n(g;\C):&=\argmin_{P\in \Pi_n}\sum_{j=1}^{M-1} \left(w_\alpha^{-1}(x_{j})g(x_{j})-P(x_{j})\right)^2 (x_{j}-x_{j+1})w_\alpha^2(x_{j})\\
 &= \argmin_{P\in \Pi_n}\sum_{j=1}^{M-1} \left(g(x_{j})-w_\alpha(x_{j})P(x_{j})\right)^2 (x_{j}-x_{j+1}).
 \end{split}
 \end{equation}

 When clear from the context, we shall omit $\C$ and write
 $S_n\left(g\right)$.

  Let
  \begin{equation}\label{eq:mrs-1st}
    a_n:=\beta_{\alpha}n^{1\over \alpha},\qquad
    \beta_{\alpha}:=\left\{\frac{2^{\alpha-2}\Gamma(\alpha/2)^2}{\Gamma(\alpha)}\right\}^{1/\alpha}.
  \end{equation}
  
 Our first main result below bounds the error $|f(z)-S_{n}(g)(z)|$ for
 $z\in\CC$ with the particular choice
 $$n=\fn(\ee,\alpha,\tau,\lambda)=\biggl\lfloor{1\over q(\ee,\alpha,\tau,\lambda)}
 \log{1\over\ee} \biggr\rfloor$$ with
 $q(\ee)\approx\log\log{1\over\ee}$, under the assumption that the
 sampling set $\C$ approximately lies in the interval
 $\left[-a_{\fn},a_{\fn}\right]$ and is sufficiently
 dense\footnote{$q(\ee)$ is given by \eqref{eq:q-def}, and the density
   conditions are given in \eqref{eq:sampling_extent_condition} and
   \eqref{eq:density_condition_main}}.

 The error bound heuristically behaves like a $z$-dependent fractional
 power of the perturbation, but in reality it is slightly more
 complicated. In more detail:
 
 \begin{enumerate}
 \item There is a natural rescaling of the $z$ variable by a factor of
   $a_{\fn}$ (which in turn depends on $\ee$), because the sampling
   set $\C$ and the resulting approximating polynomial $S_{\fn}$
   themselves depend on $\fn$. So instead of bounding
   $|f(z)-S_{n}(g)(z)|$ directly, we have a bound of the form 
   $$|f(a_{\fn}z)-S_{\fn}(g)(a_{\fn}z)|\lessapprox \ee^{\gamma(z)}.$$
 \item The exponent $\gamma(z)$ in fact has a weak dependence on
   $\ee$, and it is of the form
   $\gamma(z)=1-\frac{1}{q(\ee)}\delta(z)$ where
   $q(\ee)\approx\log\log{1\over\ee}$ is the same function used in the
   definition of $\fn$ above and $\delta(z)$ is a certain
   logarithmic potential.
 \end{enumerate}

 There are three distinct regions in the complex plane with respect to
 the error asymptotics:
 \begin{enumerate}
 \item The ``approximation region'' $\left[-a_{\fn},a_{\fn}\right]$.
   In the range $[-1,1]$ we have in fact
   $\delta(x)=|\beta_{\alpha}x|^\alpha$, and therefore (disregarding
   the rounding effects) it can be easily shown that
   $$
   \ee^{\gamma(x)}=\ee w_{\alpha}^{-1}(a_{\fn}x).
   $$
   This is the error that would be obtained by conventional
   deconvolution, i.e. division by $w_{\alpha}$.  Note that for
   any $x\in[-1,1]$ we have as $\ee\to
   0$ $$\gamma(x)=1-{\delta(x)\over q(\ee)}\to 1.$$
 \item The ``extrapolation region'', where the error
   $\ee^{\gamma(z)}$ is less than $\exp(\tau |a_{\fn}z|^{\lambda})$,
   i.e. the maximal growth rate for any function in
   $B_{\tau,\lambda}$. It turns out that the maximal extrapolation
   interval can be precisely determined as $|z|\leq r_{\fn}/a_{\fn}$, where
   $$
   r_n:=\biggl(\frac{n}{\tau\lambda}\biggr)^{\frac{1}{\lambda}}.
   $$
   In terms of the Hölder exponent $\gamma$, it turns out that
   $\lim_{\ee\to 0} \frac{\delta(r_{\fn}/a_{\fn})}{q(\ee)}=1$, and
   therefore as $\ee\to 0$ we have $$\gamma(r_{\fn}/a_{\fn})\to 0.$$
 \item The ``forbidden'' region $|z|>r_{\fn}/a_{\fn}$ where
   essentially no information can be obtained about $f$ from the
   samples on $\left[-a_{\fn},a_{\fn}\right]$.
 \end{enumerate}

 In \prettyref{fig:exp-and-bounds-scaled} we show an example for the
 behaviour of both the exponent $\gamma(z)$ and the complete bound
 $\ee^{\gamma(z)}$ for different values of $\ee$ and other parameters fixed.

 \begin{figure}[hbtp]
   \hspace*{-3em} \begingroup
   \captionsetup[subfigure]{width=0.45\columnwidth}

   \subfloat[The exponent $\gamma(z)$. As $\ee$ becomes smaller, the
   values of $\gamma$ in the interval $[0,1)$ approach 1, while at the
   right boundary $z=r_{\fn}/a_{\fn}$ they approach
   0.]{\includegraphics[width=0.5\columnwidth]{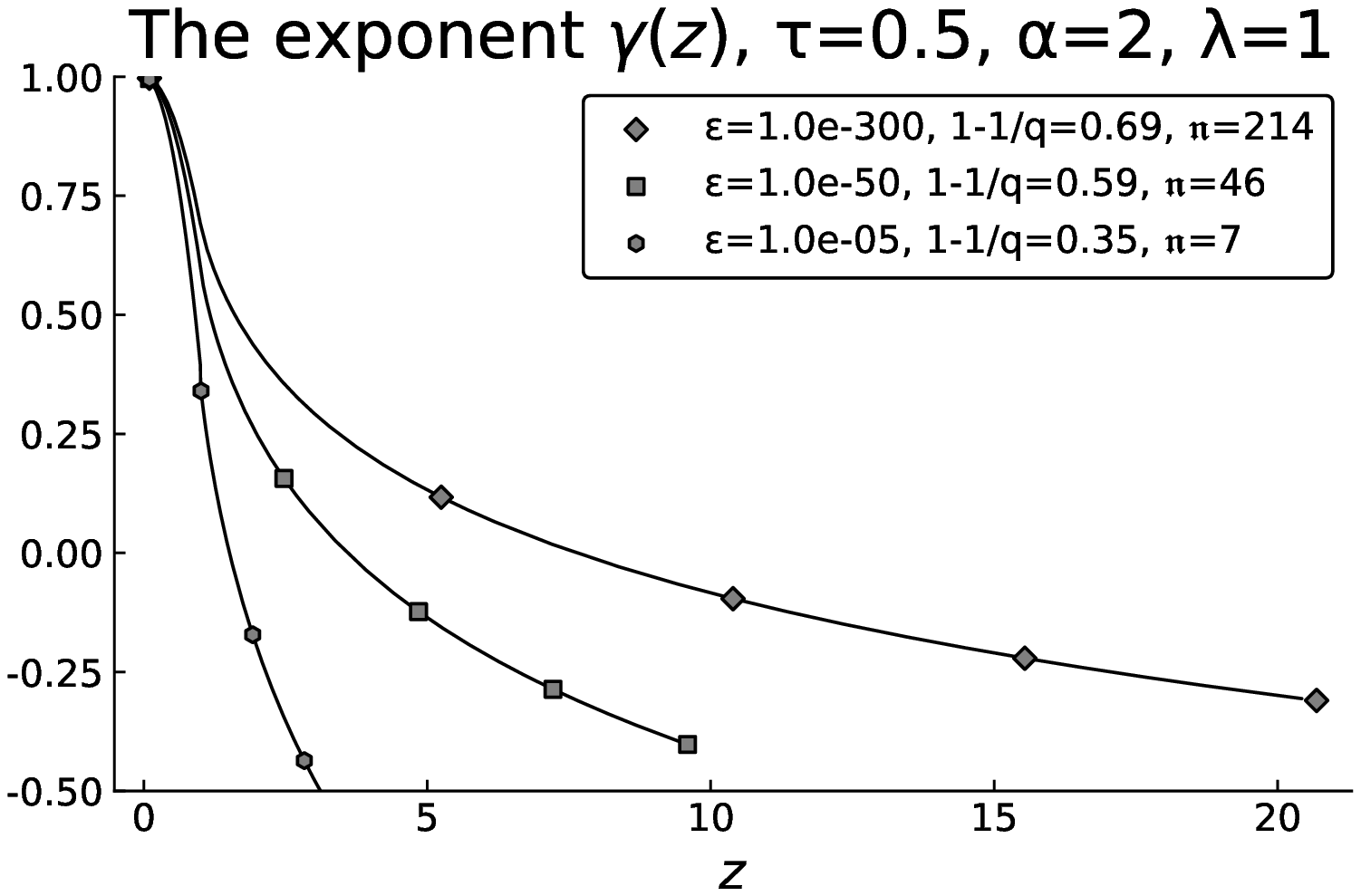}

     \label{fig:exponent-scaled}

   } \subfloat[The complete bound $\ee^{\gamma(z)}$ (solid), the
   quantity $\exp(\tau|a_{\fn} z|^\lambda)$ (dashed) and
   $w_{\alpha}^{-1}(a_{\fn}z)$
   (dotted).]{\includegraphics[width=0.5\columnwidth]{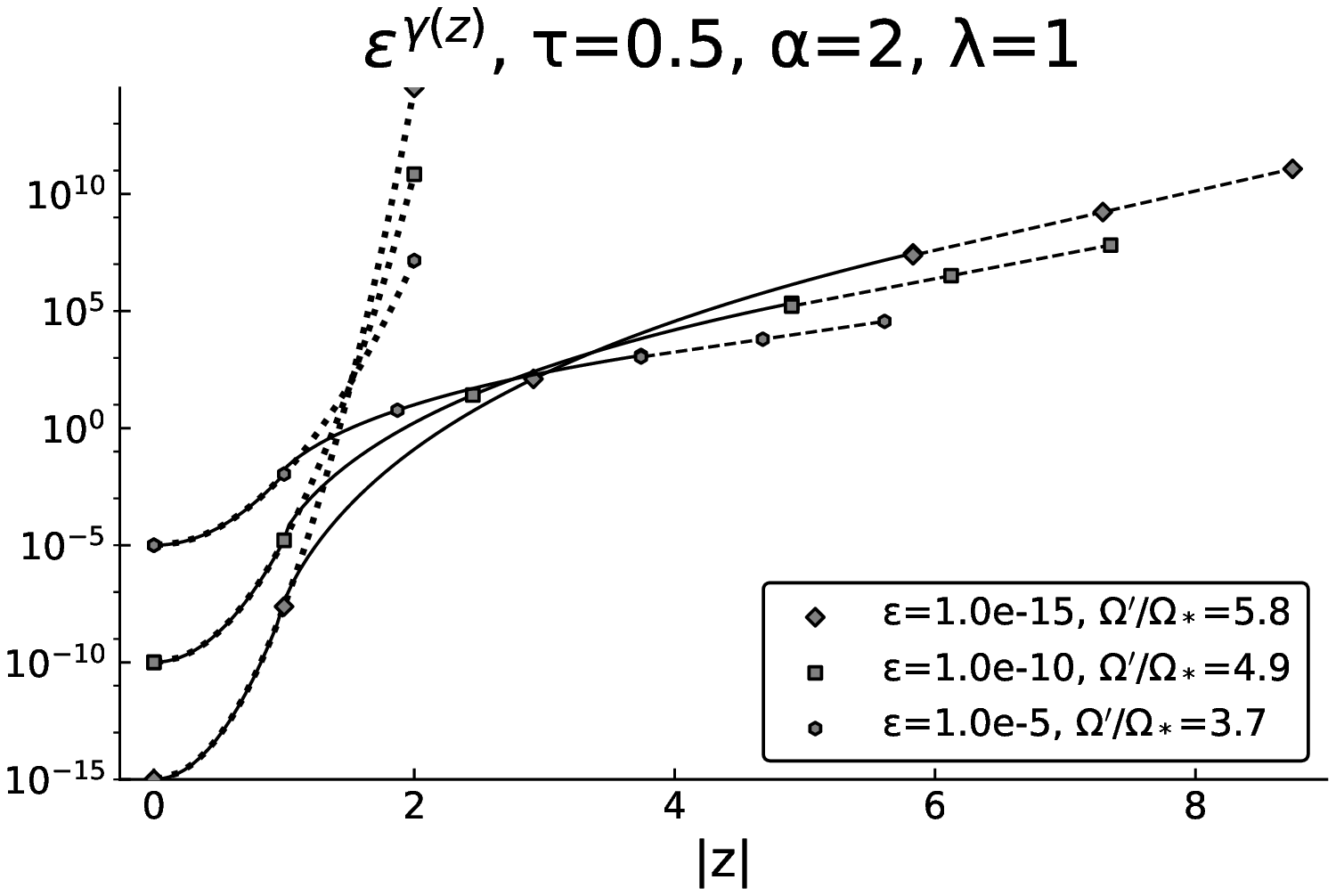}

     \label{fig:bound-scaled}

   } \endgroup
   \caption{The optimal exponent $\gamma(z)$ and the corresponding
     error bound $\ee^{\gamma(z)}$ for several values of $\ee$ and
     $\tau=0.5,\lambda=1,\alpha=2$.}

   \label{fig:exp-and-bounds-scaled}
 \end{figure}

 \begin{rem}\label{rem:asymptotic-orders}
   The length of the sampling window $\O_*$ essentially scales like
   $a_\fn\approx \left({1\over\log\log{1\over\ee}}\log{1\over\ee}\right)^{1\over\alpha}$, while the
   maximal extrapolation range $\O'$ is of the asymptotic order
   $r_\fn\approx \left({1\over
     \tau}{1\over\log\log{1\over\ee}}\log{1\over\ee}\right)^{1\over\lambda}$.  Therefore we obtain
   a genuine extension by a factor of
   $\O'/\O_*\approx
   \left({1\over\tau}\right)^{1\over\lambda}\left({1\over\log\log{1\over\ee}}\log{1\over\ee}\right)^{{1\over\lambda}-{1\over\alpha}}$.
 \end{rem}

 The result reads as follows, and is proved in \prettyref{sec:proofs}.
 
 \begin{thm}\label{thm:main-thm-exp-type}
   Let $f\in B_{\tau,\lambda}$. For $\ee>0$, $\alpha>\lambda$ let
   $\fn=\fn(\ee,\alpha,\tau,\lambda)\approx \log\frac{1}{\ee}$ be as
   defined in \eqref{eq:mstar-def} below.  For any sequence of points
   $\C=\left\{ x_M< \dots<x_1\right\}$ satisfying
   \begin{align}
     \left[-{4\over 3}a_\fn,{4\over 3}a_\fn\right]&\subseteq\left[x_M,x_1\right]\subseteq\left[-2a_\fn,2a_\fn\right]\label{eq:sampling_extent_condition},\\
     |x_j-x_{j+1}|&\leq c_1 \fn^{\frac{1}{\alpha}-1}, \label{eq:density_condition_main}
   \end{align}
   where $c_1=c_1\left(\alpha\right)$ is the explicit consant defined in
   \eqref{eq:choice-of-constants} below, let
   $S_{\fn}\left(g;\C\right)$ be the weighted least squares
   polynomial fit of degree $\fn$ to $f$ using samples of $g$ as in
   \eqref{eq:data_def}, \eqref{eq:epsbd} and \eqref{eq:lsquare_def}.

   Then the (properly rescaled) pointwise extrapolation error satisfies
   \begin{equation}\label{eq:main-extr-estimate-eps}
     \left|f\left(a_\fn z\right)-S_{\fn}(g)\left(a_\fn z\right)\right|\lessapprox
     \begin{cases}
       \ee w_{\alpha}^{-1}\left(a_\fn z\right) & z\in\left[-1,1\right] \\
       \ee^{1-\frac{1}{q\left(\ee\right)} \delta\left(z\right)} & |z| < r_\fn/a_\fn,\;z\in\CC\\
       \exp\left(\tau |a_\fn z|^\lambda\right) & |z|>r_\fn/a_\fn,\;z\in\CC
     \end{cases},
   \end{equation}
   where $q\left(\ee\right)$ is given by \eqref{eq:q-def} and satisfies
   $q(\ee)\approx\log\log{1\over\ee}$, while the function
   $\delta\left(z\right)$ is defined in \eqref{eq:expn2-def} below,
   such that
   \begin{numcases}{\lim_{\ee\to 0} \frac{\delta(z)}{q\left(\ee\right)} = }
     0 & for any fixed $|z|$\label{eq:delta-at-one}\\
     1 & for any $z=z_{\ee}$ with $|z_\ee|=r_\fn/a_\fn$.\label{eq:right-boundary-expn-asympt}
   \end{numcases}

   Furthermore, the relation in \eqref{eq:main-extr-estimate-eps}
   holds uniformly on compact subsets in $z$.

 \end{thm}

 \begin{rem}
   \label{rem:equispaced-linear}
   If the points are chosen to be equispaced, then it suffices to have
   at most linear (in the approximation degree $\fn$) growth of the
   size of the sampling set $\C$, as indeed it is sufficient to take
   $$|\C_{\textrm{equi}}|= {4\over
     3c_1}a_\fn \fn^{1-{1\over\alpha}}\leq c \fn.$$
 \end{rem}

\subsection{Optimality}
Next, we will show that the results above cannot be meaningfully
improved. To do so, we first review some facts from the theory of
optimal recovery based on \cite{micchelli1985lectures,
  micchelli1977survey}, as they relate to our problem.  We consider
the set $B_{\tau,\lambda}$ as a subset\footnote{Since the estimate
  \eref{eq:uniformity} is uniform for all functions in
  $B_{\tau,\lambda}$, $B_{\tau,\lambda}$ is a compact subset of $X$.}
of the space $X$ of all continuous functions $f :\RR\to\RR$ such that
$w_\alpha f\in L^\infty$.

Let $Y$ be a normed linear space with norm given by $\|\cdot\|_Y$, and
$I:X\to Y$, $R :Y\to X$ be functions. The function $I$ is called
\emph{information operator}, and $R$ is called a \emph{recovery
  operator}. In \cite{micchelli1985lectures, micchelli1977survey}, $I$
is required to be a linear operator, but we do not need this
restriction here. In its place, we require that
 \begin{equation}\label{eq:info_fn_cond}
 \|I(f)\|_Y \le \|w_\alpha f\|_\infty, \qquad f\in X.
 \end{equation}
 We give some examples of $R$ and $I$.
 \begin{enumerate}
 \item A trivial example is $Y=X$ and $I(f)=w_\alpha f$, $f\in X$. In
   this case the recovery operator is of interest only for an extension
   of $f$ to $\CC$.

 \item In this paper, we are considering the mapping $I :X\to\RR^M$
   given by
   \[
     I(f)=(f(x_{M})w_\alpha(x_{M}), \cdots, f(x_{1})w_\alpha(x_{1})),
   \] and the recovery operator given by $S_n$.
 \item If $w_\alpha f\in L^1+L^\infty$, we define its
   Fourier-orthogonal coefficients (see \eqref{eq:orthogonality} below)
   by
   \begin{equation}\label{eq:fourcoeffdef}
     \hat{f}(k)=\int_\RR f(t)p_k(t)w_\alpha^2(t)dt, \qquad k=0,1,\cdots.
   \end{equation}
   Let $Y$ be the space of all sequences $\{d_k\}$ for which
   $$
   \|\{d_k\}\|_Y=c_2^{-1}\sup_{k\ge 0}(k+1)^{1/(2\alpha)}|d_k| <\infty,
   $$
   where $c_2$ is defined in \eref{eq:ortho_poly_bd}.  The information
   operator $I(f)=\{\hat{f}(k)\}$ maps $X$ into $Y$, and satisfies
   \eref{eq:info_fn_cond}. An example of a recovery operator is the
   expansion $\sum_{k=0}^\infty d_kp_k(z)$ which converges when
   $f\in B_{\tau,\lambda}$ and $d_k=\hat{f}(k)$.  This shows that one
   can weaken the condition \eref{eq:info_fn_cond}, allowing a
   constant factor on the right hand side of the inequality, by
   renormalizing $Y$.
 \end{enumerate}

 For any information operator $I$ and recovery operator $R$,
 $z\in \CC$, and $\ee >0$, we are interested in the worst case
 recovery error when the information is perturbed by $\ee$:
 \begin{equation}\label{eq:worst_case_error}
 E(\tau,\lambda, \ee, z; R, I):=\sup_{f\in B_{\tau,\lambda},\;\|y\|_Y \le \ee}|f(a_\fn z)-R(I(f)+y)(a_\fn z)|,
\end{equation}
where again $\fn=\fn\left(\ee,\alpha,\tau,\lambda\right)$ as given by
\eqref{eq:mstar-def}. The best worst case (minimax) error is defined by
 \begin{equation}\label{eq:noisy_optimal_recovery}
 \mathcal{E}(\tau,\lambda, \ee,z):=\inf_{R,I}E(\tau,\lambda, \ee; R, I),
 \end{equation}
 where it is understood that the infimum is over all $R$ and $I$ with
 all normed linear spaces $Y$, subject to
 \eref{eq:info_fn_cond}. Informally,
 $\mathcal{E}(\tau,\lambda, \ee,z)$ the best accuracy one can expect
 in reconstructing $f(z)$ for every $f\in B_{\tau,\lambda}$ from any
 kind of information and recovery based only on the values of $f$ on
 $\RR$. It is closely related to the notion of ``best possible
 stablity estimate'' of K.Miller \cite{miller_least_1970},
 cf. \prettyref{subsec:novelty-related-work}.

 With these notations, the bound \eqref{eq:main-extr-estimate-eps} in
 \prettyref{thm:main-thm-exp-type} means that (with
 $q=q\left(\ee\right)$ as before)
 \[
 \mathcal{E}(\tau,\lambda, \ee,z) \lessapprox \ee^{1-{1\over q}\delta\left(z\right)},\quad |z|\leq r_\fn/a_\fn.
 \] 

 The following theorem shows that in this sense of optimal recovery, this result is the best possible.

 \begin{thm}\label{thm:lower-bound-eps}
   There exists a function $\xi(\ee)$ satisfying
   $\xi(\ee)\lessapprox \ee$, such that for any $z\in\CC$, $\ee>0$
   \begin{equation}\label{eq:optimal-lower-bound}
     \mathcal{E}\left(\tau,\lambda,\xi\left(\ee\right),z\right)\gtrapprox \ee^{1-{1\over q}\delta(z)},
   \end{equation}
   the relation holding uniformly for compact sets in the complex
   plane (w.r.t $z$).

   In particular, for any small enough $\ee\ll1$ there exists a ``dark
   object'' $f_{\ee}$ such that
   \begin{enumerate}
   \item for $x\in\left[-\O_*,\O_*\right]$ it has the same
     magnitude as the perturbation level:
     $$\left|f_{\ee}\left( x\right)\right|\lessapprox \ee
     w_{\alpha}^{-1}\left( x\right);$$
   \item outside the sampling window, it has the same
     magnitude as the extrapolation
     error:
     $$\left|f_{\ee}\left(a_\fn z\right)\right|\gtrapprox\ee^{1-{1\over q}\delta\left(z\right)},\qquad |z|\le r_\fn/a_\fn.$$
   \end{enumerate}
 \end{thm}

\section{A numerical illustration:  functions of order $\lambda=1$  with Hermite
  polynomials}\label{sec:numsect}

In this section we
specialize our preceding results to the case $\alpha=2$, $\lambda=1$,
with $f$ a function of finite exponential type $\tau>0$. We then run a
simple computational experiment and compare the results with the
theoretical predictions, showing good agreement between the two in
practice.

For technical convenience, we have chosen to work with an
off-the-shelf implementation of Hermite polynomials, which are
orthogonal with respect to the weight function
$u_2(x)=\exp\left(-{x^2\over 2}\right)$. So instead of
\eqref{eq:data_def} we assume that $f$ is blurred by
$u_2$. Since our weights $w_{\alpha}$ are not of this form, we perform
a trivial change of variable $t=x/\sqrt{2}$ and apply our results to
the function $h(t):=f(t\sqrt{2})$, which is consequently of
exponential type $\tau'=\tau\sqrt{2}$.

In this case we have $\beta_2=1$, $a_n=\sqrt{n}$, $F_2=\log(1/2)-1/2$,
$v_2(t)=(1/\pi)\sqrt{1-t^2}$ and also\footnote{According to \cite[(2.9)]{mhaskar1984extremal}, the relationship between $\delta(z)$ and the function
  $G(2;z)$ is
  $$G(2;z)=\exp\left\{\delta(z)-\log\left| z+\sqrt{z^2-1} \right| -
    |z|^2\right\}.$$ The explicit formula for $G(2;z)$ is then given
  by \cite[(2.19),(2.20)]{mhaskar1984extremal}.}
 \begin{equation}\label{eq:hermite_exponent}
   \delta(z)= U(z)-F_2=\log\left|z+\sqrt{z^2-1}\right| +\Re\left(z^2-z\sqrt{z^2-1}\right),
 \end{equation}
 where the branch of $\sqrt{z^2-1}$ is chosen so that
 $\sqrt{z^2-1}/z\to 1$ as $z\to\infty$.

Further, according to our notations, we also have
$\rho=\tau'\sqrt{e}/2$, $\mu=1/2$ and 
\begin{align*}
   q\left(\ee\right) &= {1\over 2}\lam \biggl({4\over \tau^2 e}\log{1\over\ee}\biggr),\\
   \fn &= \left\lfloor {1\over q}\log{1\over\ee}\right\rfloor.
\end{align*}

% The corresponding exponent $\expn_{\ee}$ for different values of $\ee$
% is shown in \prettyref{fig:optimal-exponent}.

 % \begin{figure}[hbtp]
 % \includegraphics[height=0.20\paperheight]{drawings/figs_latest/exponent}\caption{The
 %   optimal exponent $\expn(z)$ for $\tau=0.5$}
 % \label{fig:optimal-exponent}
 % \end{figure}

To implement the least squares operator $S_\fn$, we have chosen to
work in the basis of Hermite orthogonal polynomials $H_n$ which
satisfy $\int H_k(x) H_j(x)
\exp\left(-x^2\right)dx=\delta_{k,j}$. Following
\prettyref{rem:equispaced-linear} we pick $2\fn$ equispaced samples in
$\left[-\sqrt{2\fn},\sqrt{2\fn}\right]$ (i.e. the oversampling factor
is $2$).

For running the experiments below, we have chosen the following model
function $f_{\tau}$ to extrapolate:
\begin{equation}\label{eq:fun-example}
  f_{\tau}(x):={1\over 14}\bigl(5+\cosh\left(\tau x - 2\right) + \sinh\left(\tau x\right)\bigr).
\end{equation}
A simple computation shows that $\tn f_{\tau} \tn_{\tau,1}\leq 1$,
and therefore  $f_{\tau}\in B_{\tau,1}$.

\prettyref{thm:main-thm-exp-type} implies that the error satisfies, in
the unscaled $z$ variable,
\begin{equation}
  \label{eq:num-bounds-explicit}
  |f_{\tau}(z)-S_\fn(g)(z)|\lessapprox E_{\tau,\ee}(f;z):=
  \begin{cases}
    \ee \exp\left(z^2/2\right) & z\in \left[-\sqrt{2\fn},\sqrt{2\fn}\right]\\
    \ee^{\expn'_{\ee}(z)},\quad \expn'_{\ee}(z)=1-{1\over
      q}\delta\left({z\over\sqrt{2\fn}}\right) &
    |z|\in[\sqrt{2\fn},{\fn\over\tau}]\\
    \exp\left(\tau|z|\right) & |z|>{\fn\over\tau}.
  \end{cases}
\end{equation}

Instead of the construction used in the proof of
\prettyref{thm:lower-bound-eps} (i.e. the function $P_\fn^*$ given in
\eqref{pf5eqn2}), we shall use the following function as our ``dark
object'' (unrelated to $f_{\tau}$):
 \begin{equation}\label{eq:new-dark-object}
   f_{\tau,\ee}\left(z\right): =\cosh\left(\tau z\right)-\sum_{n<\fn}e^{\frac{\tau^{2}}{4}}\frac{\tau^{n}\pi^{\frac{1}{4}}\left(1+\left(-1\right)^{n}\right)}{\sqrt{2^{n}n!}}H_{n}\left(z\right).
 \end{equation}
 It is not difficult to show that this function is in $B_{\sigma,1}$
 and in fact also satisfies
 \eqref{pf5eqn3} and
 \eqref{eq:rescaled-lower-bound-on-compact-sets}\footnote{Proof is
   available upon request.}. The underlying reason for using a
 different function $f_{\tau,\ee}$ is that it is not known how to
 evaluate $P_\fn^*$  general, while \eqref{eq:new-dark-object} is an
 absolutely explicit formula.

 In \prettyref{fig:sigma-eps-fixed} we show the reconstruction and the
 corresponding errors + bounds for fixed $\tau,\ee$ in the original scaling. As can be seen
 in \prettyref{fig:error-and-bounds}, the derived bounds $E_{\tau,\ee}$ are
 reasonably accurate. In \prettyref{fig:function-and-extr} it is
 clearly seen that the algorithm chooses a reasonable value for
 $\fn$, avoiding the extreme noise outside the essential
 support of the window.

 In \prettyref{fig:epsilon-varying} the reconstruction and comparison
 were performed for fixed $z_0$ and $\sigma$, varying $\ee$. It can be
 seen that the dependence of the error on $\ee$ is accurately
 determined.  The threshold values of $\ee$ for which one moves from
 interpolation to extrapolation region ($\ee_{1\to2}$) and from
 extrapolation to forbidden region ($\ee_{2\to3}$) can be
 approximately determined as the solutions $\ee_{1\to 2}$ and
 $\ee_{2\to 3}$ of the equations
 \begin{align}
   z_0 &= \sqrt{2\fn} &&\iff& \ee_{1\to 2}&=\exp\left(-q(\ee_{1\to 2})z_0^2/2\right), \label{eq:region-boundaries12}
   \\
   z_0 &= r_\fn &&\iff& \ee_{2\to 3}&=\exp\left(-q(\ee_{2\to 3})\tau z_0\right). \label{eq:region-boundaries23}
 \end{align}

 \begin{figure}[h]
   \hspace*{-3em} \begingroup
   \captionsetup[subfigure]{width=0.45\columnwidth}

   \subfloat[The error and the bounds. Thin black lines are
   $|f_{\tau}-S_\fn(g)(z)|$ for different noise realizations (see
   \eqref{eq:data_def} and \eqref{eq:epsbd}), their maximal envelope
   is the red solid curve. The green, dashed curve is the analytical
   bound $E_{\tau,\ee}$ in \eqref{eq:num-bounds-explicit}, while the
   magenta, dashdotted curve is the minimax function $f_{\tau,\ee}$
   defined in \eqref{eq:new-dark-object}. The dotted vertical lines
   are the region
   boundaries.]{\includegraphics[width=0.5\columnwidth]{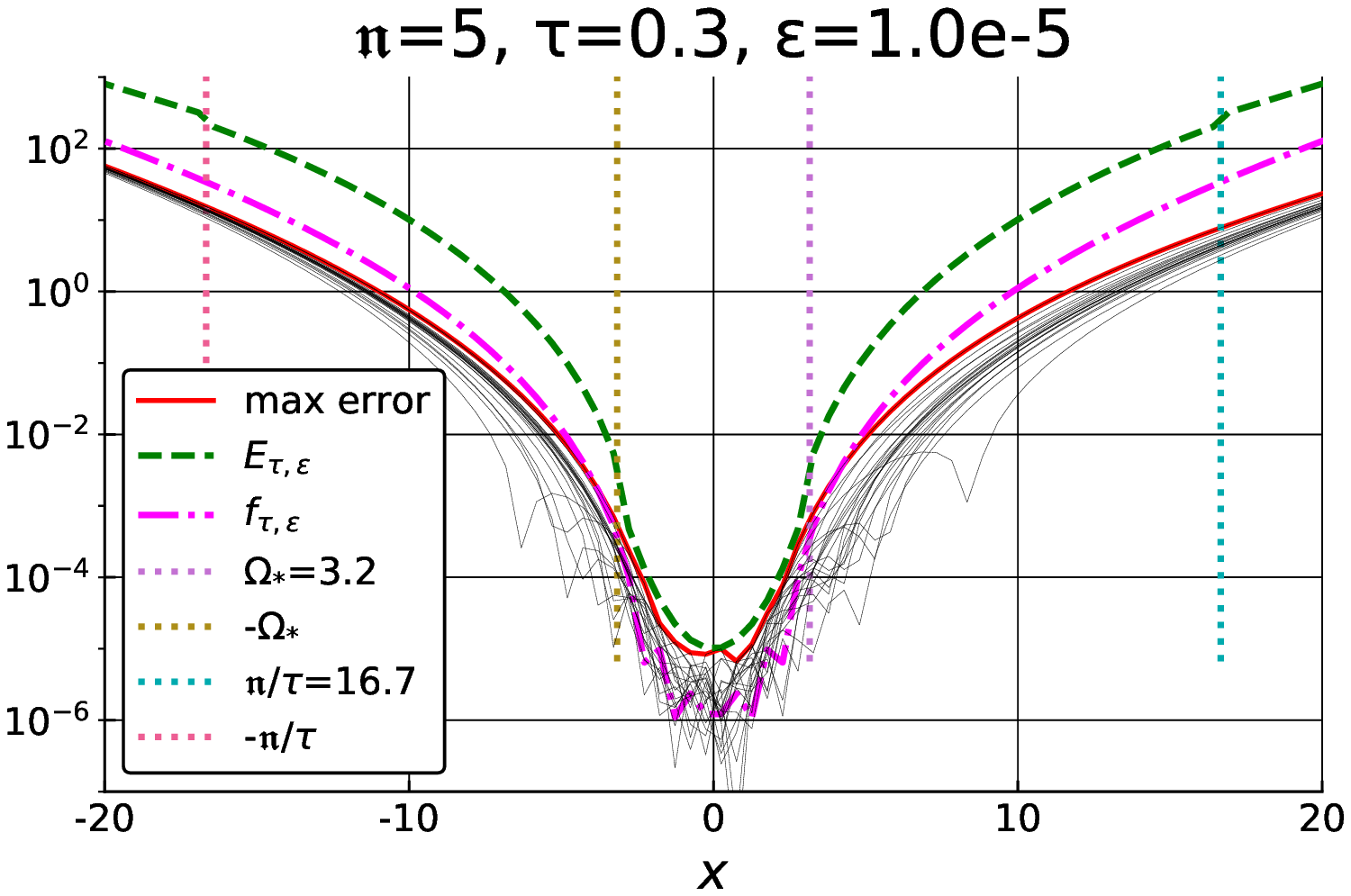}

     \label{fig:error-and-bounds}

   } \subfloat[The function (blue dashed) and its extrapolant (red
   solid). The black dots are the actual sampling points, while the
   grey curve is the noisy function $u_{2}^{-1}(x)g(x)$ which would
   need to have been used without extrapolation -- recall
   \eqref{eq:data_def}. As can be seen, the choice of $\fn$ ensures
   that the samples are taken inside the essential support of the
   window, which depends on
   $\ee$.]{\includegraphics[width=0.5\columnwidth]{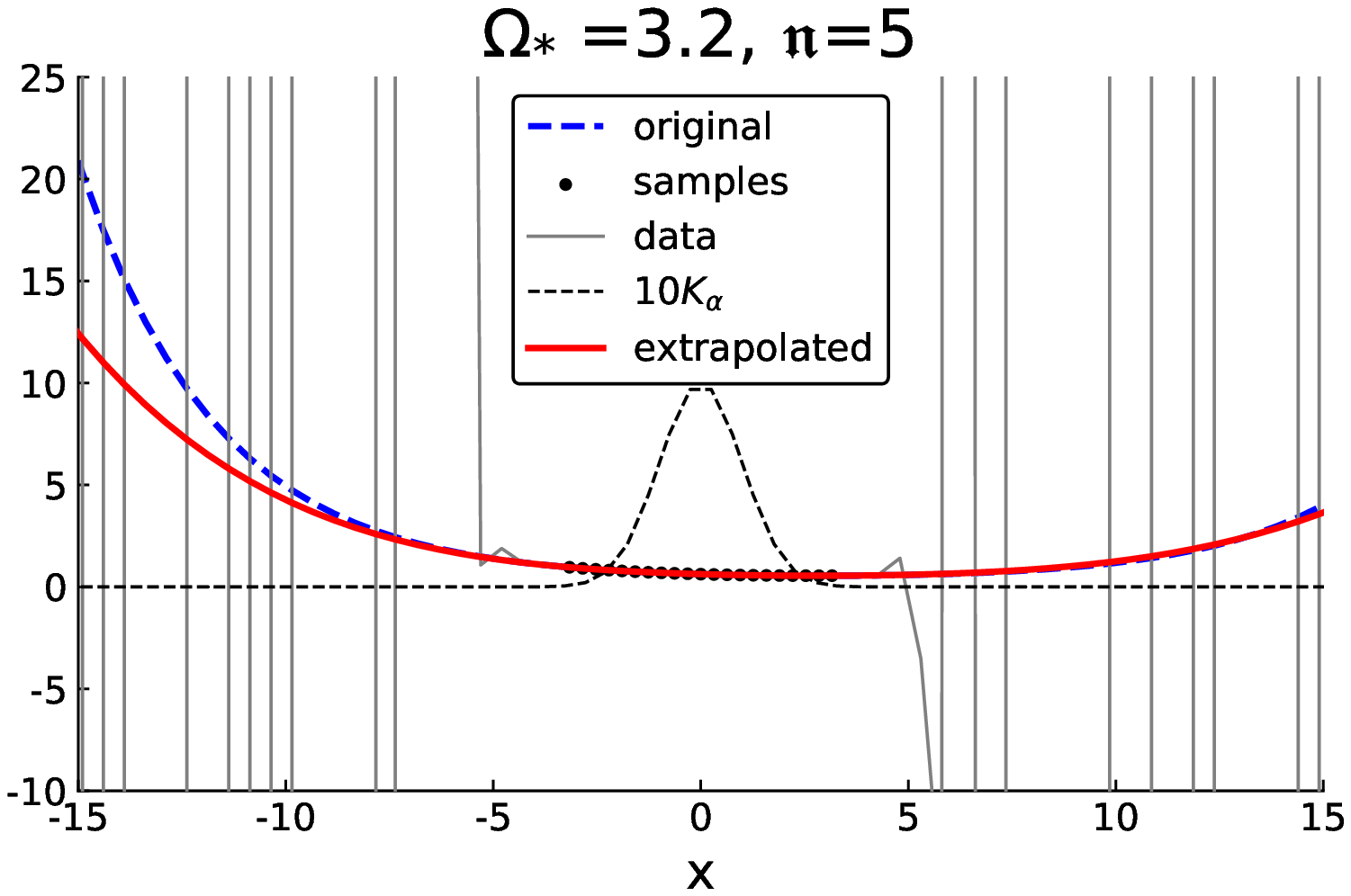}

     \label{fig:function-and-extr}

   }

   \endgroup
   \caption{Results for the function defined in \eqref{eq:fun-example}
     with $\tau=0.3$ and $\ee=10^{-5}$.}

   \label{fig:sigma-eps-fixed}
 \end{figure}

 \begin{figure}[hbtp]
   \hspace*{-3em} \subfloat[$z_0=4$, the transition between regions 1
   and 2 (black vertical line) computed numerically from
   \eqref{eq:region-boundaries12}.]{\includegraphics[width=0.55\columnwidth]{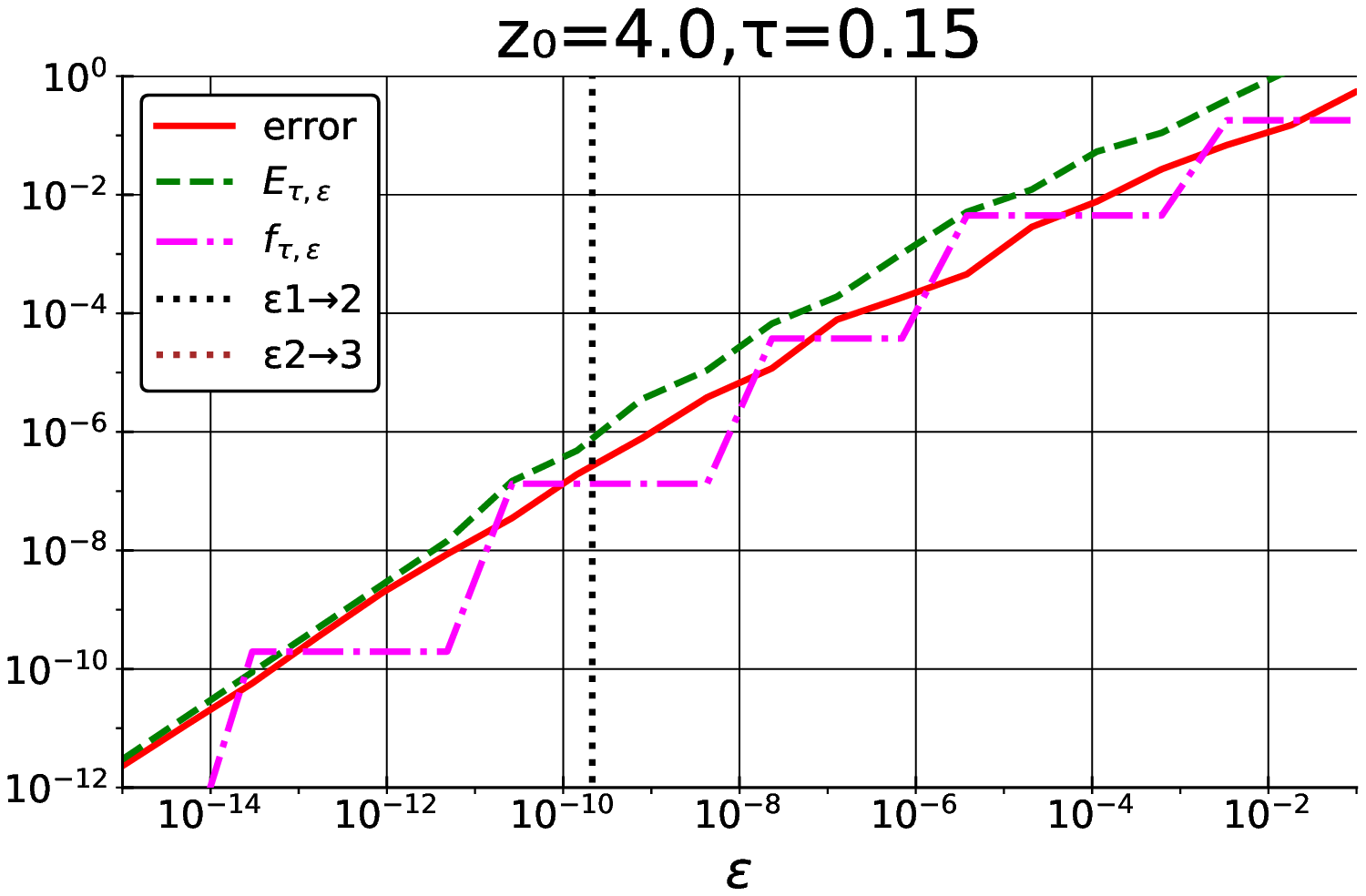}

   } \subfloat[$z_0=38$, the transition between regions 2 and 3
   (brown vertical line), computed numerically from
   \eqref{eq:region-boundaries23}.]{\includegraphics[width=0.55\columnwidth]{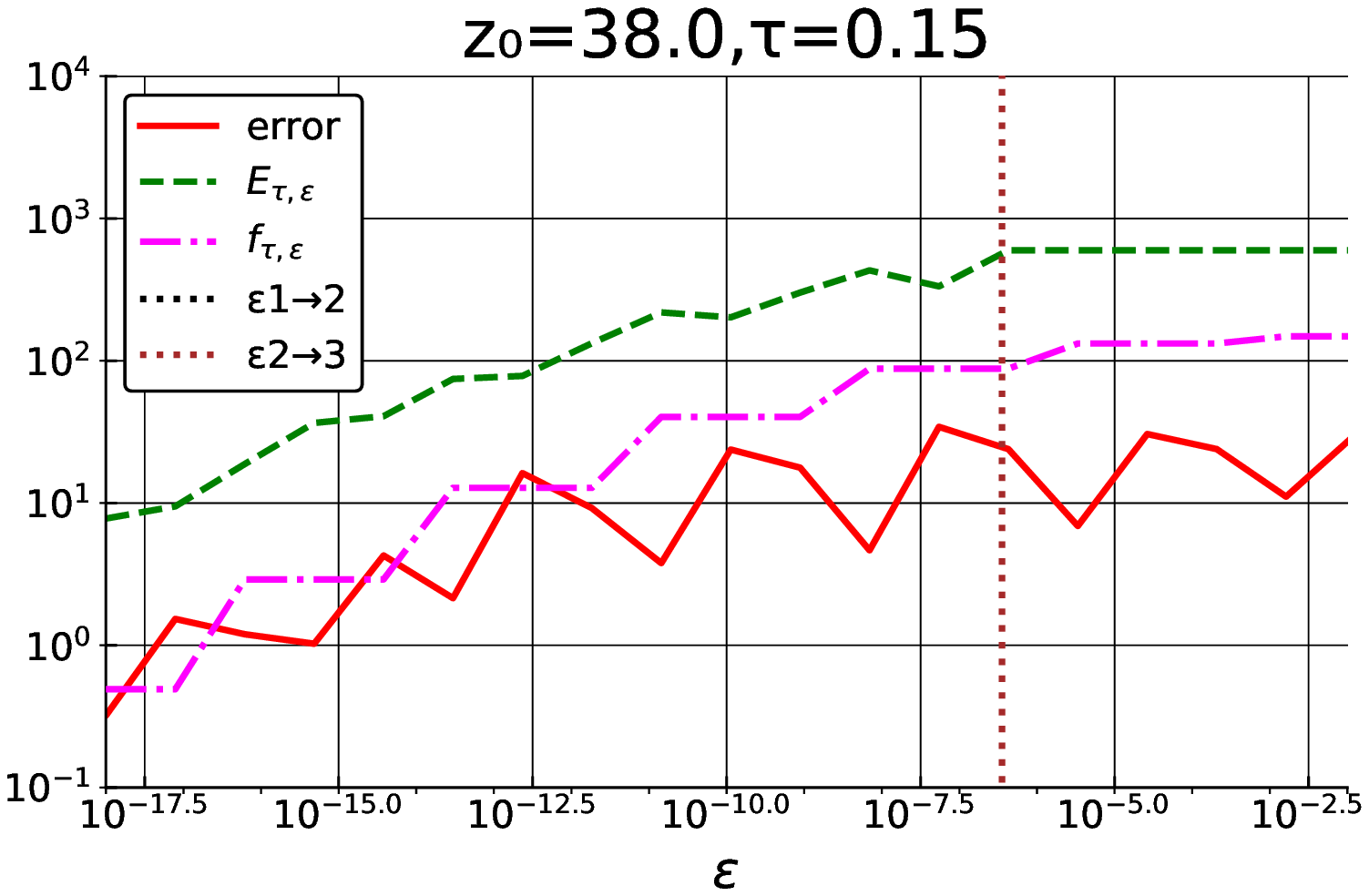}

   }

   \caption{Error (red, solid curve) vs bound $E_{\tau,\ee}$ (green,
     dashed curve) and the minimax function $f_{\tau,\ee}$ (magenta,
     dashdotted) from \eqref{eq:new-dark-object} for $\tau=0.15$, as
     function of $\ee$ for a fixed $z_0$. Same function $f_{\tau}$ as
     in \eqref{eq:fun-example}.}
   \label{fig:epsilon-varying}
 \end{figure}

     % NB
\section{Preliminaries from weighted approximation theory}\label{sec:backsect}
\subsection{Weighted polynomials}
 A very important fact in the theory of weighted approximation is that the supremum
 norm of weighted polynomials is attained on an interval depending only on the 
 degree of the polynomial, and not on the individual polynomials involved. 
 We need to describe this fact in some detail. 
 Let
 \begin{equation}\label{eq:MRSno} 
 a_x(\alpha):= \beta_\alpha
 x^{1/\alpha}:=\left\{\frac{2^{\alpha-2}\Gamma(\alpha/2)^2}{\Gamma(\alpha)}\right\}^{1/\alpha}x^{1/\alpha},
 \qquad x>0, 
 \end{equation}
 be the so called Mhaskar-Rakhmanov-Saff numbers. The Ullman distribution is defined by
 \begin{equation}\label{eq:ullmandist}
 v_{\alpha}\left(t\right):=\frac{\alpha}{\pi}\int_{\left|t\right|}^{1}\frac{y^{\alpha-1}}{\sqrt{y^{2}-t^{2}}}dy,\quad\left|t\right|\le 1.
 \end{equation}
 Further denote by $U\left(z\right)$ its logarithmic potential
 \begin{equation}\label{eq:log-potential-def}
 U\left(z\right):=U(\alpha;z)=\int_{-1}^{1}\log\left|z-t\right|dv_{\alpha}\left(t\right),\quad z\in\mathbb{C}.
 \end{equation}
 The number
 \begin{equation}\label{eq:robin_const}
 F_{\alpha}:=\log(1/2)-\frac{1}{\alpha}
 \end{equation}
 is called the modified Robin's constant for $w_{\alpha}$.

 The following is proved in \cite[Theorem~6.4.2]{mhaskar1997introduction}, in fact, for all $\alpha>0$.

 \begin{thm}\label{theo:mhas_saff_theo}
   Let $\alpha>0$. Then
   \begin{enumerate}[label=(\alph*)]
   \item The potential $U$ satisfies
     \begin{equation}\label{eq:frostman}
       \begin{split}
         U(\alpha;x)=|\beta_\alpha x|^\alpha + F_\alpha, &\mbox{ if $x\in [-1,1]$,}\\
         U(\alpha;x)< |\beta_\alpha x|^\alpha + F_\alpha, &\mbox{ if
           $x\in \RR\setminus [-1,1]$.}
       \end{split}
     \end{equation}
   \item If $n>0$, $P\in\Pi_n$, then
     \begin{equation}\label{eq:Pmain_int_est}
       |P(z)| \le \exp\left(nU(\alpha;z/a_n)-nF_\alpha\right)\max_{x\in [-a_n,a_n]}|w_\alpha P(x)|.
     \end{equation}
     In particular,
     \begin{equation}\label{eq:sup_range_ineq}
       \|w_\alpha P\|_\infty = \max_{x\in [-a_n,a_n]}|w_\alpha(x) P(x)|.
     \end{equation}
   \end{enumerate}
 \end{thm}
 The following theorem gives an analogue of \eqref{eq:sup_range_ineq} in the case of the $L^p$ norms.
 The proof of this theorem is essentially in \cite{mhaskar1997introduction}, but proved in \cite[Theorem~2.2]{mhaskar2002onthe} in the form given below.

 \begin{thm}\label{theo:summarytheo} 
 Let $1\le p<\infty$, $\alpha>1$, $a_x$ be as in
 \eref{eq:MRSno}.  For any $\eta>0$, integer $m\ge c\log(1/\eta)$ and $P\in \Pi_m$,
 we have 
 \begin{equation}\label{eq:aroundendpt} 
 \int_{|y|\ge a_m(\alpha)(1+A(\alpha,\eta) /(pm)^{2/3})}
 |P(y)w_\alpha(y)|^pdy \le \eta\|w_\alpha P\|_p^p, 
 \end{equation} 
 where 
 \begin{equation}\label{eq:lubinsky_const} 
 A(\alpha,\eta) \ge
 \left\{\left(\frac{2}{3}\alpha\min (2^{\alpha-2},
 1/(\alpha-1))\right)^{-1}\left(\log(2/(2-\sqrt{3})^2)+\log(1/\eta)\right)\right\}^{2/3}.
 \end{equation} 
 \end{thm}
 In the sequel, we will use the notation
 \begin{equation}\label{eq:Deltadef}
 \Delta_n(p,\alpha,\eta)=[-a_n(\alpha)(1+A(\alpha,\eta) /(pn)^{2/3}), a_n(\alpha)(1+A(\alpha,\eta) /(pn)^{2/3})].
 \end{equation}

\begin{defn} Let $\mathcal{T}_{n}$ denote the unique monic polynomial of degree $n$ satisfying
 \begin{equation}\label{eq:cheb_def}
 \|\mathcal{T}_{n}w_{\alpha}\|_{\infty}=\inf_{P\in\Pi_{n-1}}\|\left(\left(\cdot\right)^{n}-P\right)w_{\alpha}\|_{\infty}.
 \end{equation}
These are also called weighted Chebyshev polynomials, as the above
expression is completely analogous to the well-known property of the
classical Chebyshev polynomials $T_n$, namely that $T_n$ is the
minimax monic approximant to the zero function \cite[Theorem 3.6]{gil2007numerical}.
\end{defn}

The following proposition (see e.g.\cite[Section~6.3]{mhaskar1997introduction}, \cite[Corollary~3.3]{lubinsky1988strongasymptotics}) lists some of the required properties of $\mathcal{T}_n$. 

 \begin{prop}
   \label{prop:chebyshev_prop}
   Let $\alpha \ge 2$. Then
   \begin{enumerate}[label=(\alph*)]
   \item For $n\ge 1$, the polynomial $\mathcal{T}_n$ has $n$ simple zeros in
     $[-a_n,a_n]$:     $-x_{n,n}^*<\cdots<x_{1,n}^*$. 
   \item We have
     \begin{equation}\label{eq:cheb_norm_limit}
       \lim_{n\to\infty}\frac{\|\mathcal{T}_{n}w_{\alpha}\|_{\infty}}{a_n^n\exp(-nF_\alpha) }=(1/2)e^{-1/\alpha}.
     \end{equation}
   \item Uniformly on compact subsets of $\CC\setminus [-1,1]$, we have
     \begin{equation}\label{eq:cheb_ext_limit}
       \lim_{n\to\infty} \frac{|\mathcal{T}_{n}(a_nz)|}{a_n^n\exp(-nU(z))}= \lim_{n\to\infty} \frac{|\mathcal{T}_{n}(a_nz)|}{\|\mathcal{T}_{n}w_{\alpha}\|_{\infty} \exp(-nU(z)+nF_\alpha)} =1.
     \end{equation}
   \end{enumerate}
 \end{prop}

 We will also use another family of polynomials. Let $\{p_k\}_{k=0}^\infty$ be the system of orthonormalized polynomials with respect to $w_\alpha^2$; i.e., for each $k=0,1,\cdots$, $p_k(x)=\gamma_k x^k
 + \cdots\in \Pi_k$, and for $k, j=0,1,\cdots$,
 \begin{equation}\label{eq:orthogonality}
 \int_\RR p_k(t)p_j(t)w_\alpha^2(t)dt =\left\{\begin{array}{ll}
 1, &\mbox{ if $k=j$,}\\
 0, &\mbox{otherwise.}
 \end{array}\right.
 \end{equation}
 It is known \cite[Theorem~13.6]{levinlubinsky} that there exist constants $c_1, c_2>0$ depending only on $\alpha$ such that
 \begin{equation}\label{eq:ortho_poly_bd}
 c_1\le n^{1/(2\alpha)}\|w_\alpha p_n\|_1 \le c_2.
 \end{equation}

\subsection{Weighted approximation of entire functions}

 If $1\le p\le\infty$ and
 $w_\alpha f\in L^p$, we define the degree of approximation of $f$ by
 \begin{equation}\label{eq:deg_approx_def} 
 E_n(p;f)=E_n(\alpha, p;f) := \inf_{P\in\Pi_n}\| (f-P)w_\alpha\|_p, \qquad n=0,1,2,\cdots.
 \end{equation} 

 In
 \cite[Theorem~7.2.1(b)]{mhaskar1997introduction}, we have proved the
 following theorem (with a different notation). 

 \begin{thm}\label{theo:charact} 
 Let $1\le p\le \infty$, $\l>0$, $\alpha>\lambda$, $w_\alpha
 f\in L^p$, and 
 \begin{equation}\label{eq:rho1def}
 \rho_1(\alpha,p;f):=\limsup_{n\to\infty}\left\{n^{n/\lambda-n/\alpha}E_n(\alpha,p;f)\right\}^{1/n}<\infty.
 \end{equation} 
 Then $f$ has an extension to the complex plane as an entire function of
 order $\lambda$ and type $\tau$ given by 
 \begin{equation}\label{eq:sigmadef} 
 \rho_1(\alpha,p;f)=(\beta_\alpha/2)(\tau \lambda)^{1/\lambda}\exp(1/\lambda-1/\alpha) (=:\rho(\alpha,\tau,\lambda)).
 \end{equation} 
 Conversely, if $f$ is
 the restriction to the real line of an entire function of order $\lambda$ and type
 $\tau$, then $w_\alpha f\in L^p$ for every $p$, $1\le p\le \infty$, $\rho_1(\alpha,p;f)$
 defined by \eref{eq:rho1def} is finite, and \eref{eq:sigmadef} holds.
 \end{thm}

                                                                 % MT
\section{Proofs}
\label{sec:proofs}

\subsection{Asymptotics as $n\to\infty$}

  In the sequel, the symbols $c, c_1,\cdots$ will denote positive
 constants depending on $\alpha,\tau,\lambda$, and other explicitly indicated quantities
 only.  Their values may be different in different occurrences, even within the
 same formula.

% The inequalities $cA\le B\le c_1A$ will be denoted by $A\sim B$.
We shall be using the following notation to compare decay of sequences
up to sub-exponential factors.
 \begin{defn}
   \label{def:asypmt-not-n}

   For sequences $\{A_n\}$, $\{B_n\}$, we write $A_n\precapprox B_n$
   (or $B_n\succapprox A_n$) to denote the fact that
there exists a sequence $\{M_n\}$ with
   $\lim_{n\to\infty}M_n^{1\over n}=1$, the limit being uniform in
   $\alpha,\tau,\lambda$ such that
   $$
   A_n \leq M_n B_n.
   $$
   Equivalently, for any $\delta>0$, there exists $c>0$, depending on
   $\alpha,\tau,\lambda,\delta$ and other explicitly indicated
   quantities only such that
   $$
   A_n\le c\cdot (1+\delta)^n B_n,\qquad n\in\NN.
   $$
 \end{defn}

 \prettyref{prop:asymptotic-order-eps} below relates the above
 condition to \prettyref{def:asymptotics-eps}.

 %}%

 Recall the weighted Chebyshev polynomials $\mathcal{T}_n$ from
\eqref{eq:cheb_def} and \prettyref{prop:chebyshev_prop}. The following
theorem is proved implicitly in the course of the  proof of Lemma~7.2.5 in \cite{mhaskar1997introduction}, but we will sketch a proof again since it is not stated explicitly there.

 \begin{thm}\label{theo:compact_entire_theo}
 Let $\alpha\ge 2$, $0<\lambda<\alpha$, $\tau>0$, and $f\in B_{\tau, \lambda}$. 
 For integer $n\ge 1$ let $L_n(f)$ be the unique polynomial 
 of Lagrange interpolation in $\Pi_{n-1}$ that satisfies $f(x_{k,n}^*)= L_n(f)(x_{k,n}^*)$ at each of the zeros $x_{k,n}^*$ of $\mathcal{T}_n$, $k=1,\cdots,n$.
 Let $\{ b_k\}$ be a sequence such that 
 \begin{equation}\label{eq:bkcond}
 \lim_{k\to\infty} \frac{a_k}{b_k}=0.
 \end{equation}
 \begin{enumerate}
 \item If $z\in\CC$, $|z|\le b_n$, then
 \begin{equation}\label{eq:inter-error-mid}
 \left|f(z)-L_n(f)(z)\right| \precapprox \left|\mathcal{T}_{n}\left(z\right)\right|b_n^{-n}\exp(\tau b_n^\lambda)
 \precapprox \exp(nU(z/a_n)-nF_\alpha)b_n^{-n}\exp(\tau b_n^\lambda).
 \end{equation}
 \item If  $z\in\CC$, $|z|>b_n$, then
 \begin{equation}\label{eq:inter-error-large}
 \left|f(z)-L_n(f)(z)\right| \precapprox \exp(\tau |z|^\lambda).
 \end{equation}
 \item In particular,
 \begin{align}
   E_n(\infty,f) &\le \|(f-L_n(f))w_\alpha\|_\infty \label{eq:uniformity}\\
   &\precapprox \frac{\rho(\alpha,\tau,\lambda)^n}{n^{n/\lambda-n/\alpha}}\label{eq:uniformity2}.
 \end{align}
\end{enumerate}
The relations hold uniformly on compact subsets in $z$.

\end{thm}
 \begin{proof}[Proof of \prettyref{theo:compact_entire_theo}]
% \textsc{Proof of Theorem~\ref{theo:compact_entire_theo}.}\\
 Since all the zeros of $\mathcal{T}_{n}$
 are in $[-a_n,a_n]$, we obtain for $\zeta\in\CC$, $|\zeta|\ge b_n/2$ that
 \begin{eqnarray}\label{pf3eqn1}
 \left|\mathcal{T}_n(\zeta)\right|&=&\left|\prod_{k=1}^n (\zeta-x_{k,n}^*)\right|\ge |\zeta|^n\left(1-\frac{2a_{n}}{b_n}\right)^{n}\succapprox |\zeta|^n\\
 \left|\mathcal{T}_n(\zeta)\right|&\le & |\zeta|^n\left(1+\frac{2a_{n}}{b_n}\right)^{n}\precapprox |\zeta|^n.
 \end{eqnarray}
 These hold uniformly in compact subsets in $z$.
 We now choose $r$ satisfying $(1+1/n)b_n\le r \le (1+2/n)b_n$. Then for $|z|\le b_n$, the standard formula for the error in Lagrange interpolation (cf. \cite[P.~50, formula (4)]{walsh1935interpolation}) states that
 \begin{equation}\label{pf3eqn2}
 f(z)-L_n(f)(z)=\frac{\mathcal{T}_n(z)}{2\pi i}\oint_{|\zeta|=r}\frac{f(\zeta)}{\mathcal{T}_n(\zeta)(\zeta-z)}d\zeta.
 \end{equation}
 If $|\zeta|=r$, then the definition of $B_{\tau,\lambda}$ and \eref{pf3eqn1}
 yield
 \begin{equation}\label{pf3eqn3}
 |f(\zeta)|\le \exp(\tau r^\lambda)\precapprox \exp(\tau b_n^\lambda), \qquad 
 |\mathcal{T}_n(\zeta)|\succapprox r^n \succapprox (1+1/n)^n b_n^n, \quad |\zeta-z| \ge 2b_n/n.
 \end{equation}
 Therefore, we deduce using \eref{pf3eqn2} that
 \begin{eqnarray}\label{pf3eqn4}
 \left|f(z)-L_n(f)(z)\right| & \le &\frac{\left|\mathcal{T}_{n}\left(z\right)\right|}{2\pi}\oint_{\left|\zeta\right|=r}\frac{\left|f\left(\zeta\right)\right|\left|d\zeta\right|}{\left|\mathcal{T}_{n}\left(\zeta\right)\right|\left|\zeta-z\right|}\nonumber\\
 &\precapprox& \disp \frac{\left|\mathcal{T}_{n}\left(z\right)\right|}{\pi}\cdot\frac{n}{b_n}\cdot b_n^{-n}\left(1+\frac{1}{n}\right)^{-n}\exp(\tau b_n^\lambda) \oint_{\left|\zeta\right|=r}\left|d\zeta\right|\nonumber\\
 &\precapprox& \disp \left|\mathcal{T}_{n}\left(z\right)\right|b_n^{-n}\exp(\tau b_n^\lambda).
 \end{eqnarray}
 This completes the proof of the first inequality in
 \eref{eq:inter-error-mid}. Since \eqref{pf3eqn3} holds for all
 $|z|\leq b_n$, the final inequality above holds uniformly in $z$ as well. The second inequality follows from Proposition~\ref{prop:chebyshev_prop} and Theorem~\ref{theo:mhas_saff_theo}.

 Next, if $|z|\ge b_n$, then we use the same argument as in
 \eref{pf3eqn4} with $|z|(1+1/n)\le r\le |z|(1+2/n)$, using also
 \eref{pf3eqn1} to obtain \eref{eq:inter-error-large}, uniformly on
 compact subsets in $z$.

 Next, let $r_n$ be defined as in \eref{eq:rn-def}, i.e.,
 $$
 r_{n}:=\left(\frac{n}{\tau\lambda}\right)^{\frac{1}{\lambda}}.
 $$
 Since $\lambda<\alpha$, the condition \eref{eq:bkcond} is satisfied with $r_n$ in place of $b_n$. So, for $|z|\le r_n$,
 \begin{equation}\label{pf3eqn5}
 |f(z)-L_n(f)(z)| \precapprox |\mathcal{T}_n(z)|r_n^{-n}\exp(\tau r_n^\lambda)=|\mathcal{T}_n(z)|\left(\frac{n}{\tau\lambda e}\right)^{-\frac{n}{\lambda}}.
 \end{equation}
 In particular, this estimate holds for $x\in [-r_n,r_n]$ replacing $z$, so that using \eref{eq:cheb_norm_limit} we deduce that for $x\in [-r_n,r_n]$,
 \begin{equation}\label{pf3eqn6}
 |(f(x)-L_n(f)(x))w_\alpha(x)| \precapprox \|w_\alpha \mathcal{T}_n\|_\infty \left(\frac{n}{\tau\lambda e}\right)^{-\frac{n}{\lambda}} \precapprox (\beta_\alpha/2)^n (n/e)^{1/\alpha-1/\lambda}(\tau\lambda)^{-n/\lambda}.
 \end{equation}
 A telescopic series argument as in
 \cite[Lemma~7.2.4]{mhaskar1997introduction} then leads to \eqref{eq:uniformity2}. % assertion in Theorem~\ref{theo:compact_entire_theo}.
% \qed
 \end{proof}

The main result of this subsection, and the core estimate for proving \prettyref{thm:main-thm-exp-type} is the
following.

 \begin{thm}\label{theo:maintheo}
 Let $n\ge 1$, $M\ge 2$ be integers, $\alpha\ge 2$, $\C_n=\{x_{M,n} <x_{M-1,n}<\cdots<x_{1,n}\}\subset\RR$, and $ \Delta_n(2,\alpha,1/8) \subseteq [x_{M,n},x_{1,n}]\subseteq [-2a_n,2a_n]$. We assume further that \eref{eq:meshnorm} is satisfied. Let $\tau>0$, $0<\lambda<\alpha$, and $f\in B_{\tau,\lambda}$. Then
 \begin{equation}\label{eq:ls_square_approx_uniform}
 \|(f-S_n(g;\C_n))w_\alpha\|_\infty \le cn\{E_n(\infty,f)+\ee\} \precapprox \left(\frac{\rho(\alpha,\tau,\lambda)^n}{n^{n/\lambda-n/\alpha}}+\ee\right).
 \end{equation}
 With
 \begin{equation}\label{eq:rn-def}
 r_{n}:=\left(\frac{n}{\tau\lambda}\right)^{\frac{1}{\lambda}},
 \end{equation}
 we have, uniformly on compact subsets in $z$,

 \begin{numcases}{|f(z)-S_n(g;\C_n)(z)| \precapprox}
   \left(\frac{\rho(\alpha,\tau,\lambda)^n}{n^{n/\lambda-n/\alpha}}+\ee\right)\exp(nU(z/a_n)-nF_\alpha)
   & for   $|z|\le r_n$, \label{eq:ls_ext_mid} 
 \\
 \exp(\tau|z|^\lambda)\left(1+\frac{n^{n/\lambda-n/\alpha}}{\rho(\alpha,\tau,\lambda)^n}\ee\right)
 & for $|z|>r_n$. \label{eq:ls_ext_large}
\end{numcases}

 %Equivalently by rescaling $z$, we obtain
 %\begin{numcases}{|f(a_nz)-S_n(g;\C_n)(a_nz)| \precapprox }
 %\left(\frac{\rho(\alpha,\tau,\lambda)^n}{n^{n/\lambda-n/\alpha}}+\ee\right)\exp(nU(z)-nF_\alpha),
%  & for $|z|\le r_n/a_n$,\label{eq:ls_ext_mid_alt} \\
% \exp\left(\tau|a_n z|^\lambda\right)\left(1+\frac{n^{n/\lambda-n/\alpha}}{\rho(\alpha,\tau,\lambda)^n}\ee\right),
% & for $|z|>r_n/a_n$. \label{eq:ls_ext_large_alt}
% \end{numcases}
\end{thm}

 %\begin{rem}\label{rem:conn_remark}
 %{\rm
 %The estimate \eref{eq:ls_square_approx_uniform} shows that
 %$$
 %|f(a_nz)-S_n(g;\C_n)(a_nz)| \precapprox \left(\frac{\rho(\alpha,\tau,\lambda)^n}{n^{n/\lambda-n/\alpha}}+\ee\right)
 %\exp\left((a_n|z|)^\alpha\right), \qquad z\in\RR.
 % $$
 %In view of \eref{eq:frostman}, this estimate is consistent with \eref{eq:ls_ext_mid_alt} when $z\in [-1,1]$. 
 %When $z$ is far away from $[-1,1]$, the estimate is better. 
 %For example, we may use \cite[Corollary~6.2.7]{mhaskar1997introduction} to deduce from \eref{eq:ls_ext_mid_alt} that
 % $$
 %|f(a_nz)-S_n(g)(a_nz)| \precapprox \left(\frac{\rho(\alpha,\tau,\lambda)^n}{n^{n/\lambda-n/\alpha}}+\ee\right)\exp\left(-cn(z-1)^{3/2}\right)\exp\left((a_n|z|)^\alpha\right), \qquad z\in [1,2].
% $$
% Similarly, it is not difficult to estimate that \eref{eq:ls_ext_mid_alt} gives the same estimate as \eref{eq:ls_ext_large_alt} when $z=r_n/a_n$.
% }
% \end{rem}

 Our first goal is to to prove the following estimate on $S_n(g)$. 
 \begin{thm}\label{theo:ls_square_theo}
 Let $n\ge 1$, $M\ge 2$ be integers, $\C_n=\{x_{M,n} <x_{M-1,n}<\cdots<x_{1,n}\}\subset\RR$, and $[x_{M,n},x_{1,n}]\supseteq \Delta_n(2,\alpha,1/8)$. There exists $C=C(\alpha)>0$ such that if 
 \begin{equation}\label{eq:special_meshnorm}
 \max_{1\le j\not= k\le M-1} |x_{j,n}-x_{j+1,n}|\le \frac{C}{n^{1-1/\alpha}},
 \end{equation}
 then
 \begin{equation}\label{eq:ls_square_approx_bd}
 \|(f-S_n(g;\C_n))w_\alpha\|_\infty \le c(x_{1,n}-x_{M,n})n^{1-1/\alpha}\{E_n(\infty,f)+\ee\}.
 \end{equation}
 \end{thm} 

 The first step in the proof of this theorem is the so called Marcinkiewicz-Zygmund inequality. 

 \begin{thm}\label{theo:mztheo}
 Let $n\ge 1$, $M\ge 2$ be integers, $1\le p<\infty$, $\eta>0$, $x_{M,n} <x_{M-1,n}<\cdots<x_{1,n}$, and $[x_{M,n},x_{1,n}]\supseteq \Delta_n(p,\alpha,\eta/2)$
 There exists $c=c(\alpha)>0$ such that if 
 \begin{equation}\label{eq:meshnorm}
 \max_{1\le j\not= k\le M-1} |x_{j,n}-x_{j+1,n}|\le \frac{c}{pn^{1-1/a}}\eta,
 \end{equation}
 then for every $P\in\Pi_n$,
 \begin{equation}\label{eq:lpmzineq}
 \left|\int_\RR |w_\alpha(t)P(t)|^pdt -\sum_{j=1}^{M-1}(x_{j,n}-x_{j+1,n}) |w_\alpha(x_{j,n})P(x_{j,n})|^p\right| \le \eta\|w_\alpha P\|_p.
 \end{equation}
 \end{thm}

 The proof depends upon the following Bernstein-type inequality, which
 is easy to deduce from \cite[Corollary 3.4.3, Lemma 3.4.4]{mhaskar1997introduction}.
 \begin{prop}
 \label{prop:bernsteinprop} 
 Let $1\le p\le\infty$, $\alpha>1$.  Then for every
 integer $n\ge 1$ and $P\in \Pi_n$, 
 \begin{equation}\label{eq:bernstein} 
 \|(w_\alpha P)'\|_p \le
 cn^{(\alpha-1)/\alpha}\|w_\alpha P\|_p.
 \end{equation} 
 \end{prop} 

 \begin{proof}[Proof of \prettyref{theo:mztheo}]

 %\noindent
 %\textsc{Proof of Theorem~\ref{theo:mztheo}.}\\

 Let $P\in\Pi_n$.  Without loss of generality, we may assume that
 $\|w_\alpha P\|_p=1$. In this proof we write
 $$
 \delta=\max_{1\le j\not= k\le M-1} |x_{j,n}-x_{j+1,n}|.
 $$

 Using Theorem~\ref{theo:summarytheo}, and the fact that $[x_{M,n},x_{1,n}]\supseteq \Delta_n(p,\alpha,\eta/2)$,
 we obtain that
 \begin{equation}\label{pf1eqn3} 
 \int_{t\not\in [x_{M,n},x_{1,n}]}|w_\alpha(t)P(t)|^pdt \le \eta/2.
 \end{equation}
 Next, using H\"older inequality and Proposition~\ref{prop:bernsteinprop}, we observe that
 \begin{eqnarray}\label{pf1eqn1}
 \int_{x_{M,n}}^{x_{1,n}}|w_\alpha(u)P(u)|^{p-1}|(w_\alpha P)'(u)|du &\le& \left\{\int_{x_{M,n}}^{x_{1,n}}|w_\alpha(u)P(u)|^p du\right\}^{1/p'}\left\{\int_{x_{M,n}}^{x_{1,n}}|(w_\alpha P)'(u)|^p du\right\}^{1/p}
 \nonumber\\
 &\le& cn^{1-1/\alpha}\|w_\alpha P\|_p^p=cn^{1-1/\alpha}.
 \end{eqnarray}
 Therefore,
 \begin{eqnarray}\label{pf1eqn2}
 \lefteqn{\left|\int_{x_{M,n}}^{x_{1,n}}|w_\alpha(t)P(t)|^p dt -\sum_{j=1}^{M-1}(x_{j,n}-x_{j+1,n}) |w_\alpha(x_{j,n})P(x_{j,n})|^p\right|}\nonumber\\
 &\le& \sum_{j=1}^{M-1}\int_{x_{j+1,n}}^{x_{j,n}}\left|\ |w_\alpha(t)P(t)|^p- |w_\alpha(x_{j,n})P(x_{j,n})|^p\ \right|dt\nonumber\\
 &\le& p \sum_{j=1}^{M-1}\int_{x_{j+1,n}}^{x_{j,n}}\int_{x_{j+1,n}}^{x_{j,n}}|w_\alpha(u)P(u)|^{p-1}|(w_\alpha P)'(u)|dudt\nonumber\\
 &\le& p\delta \int_{x_{M,n}}^{x_{1,n}}|w_\alpha(u)P(u)|^{p-1}|(w_\alpha P)'(u)|du\nonumber\\
 &\le& cp\delta n^{1-1/\alpha}.
 \end{eqnarray}
 Thus, if $\delta$ satisfies $cp\delta n^{1-1/\alpha} \le \eta/2$, then
 $$
 \left|\int_{x_{M,n}}^{x_{1,n}}|w_\alpha(t)P(t)|^p dt -\sum_{j=1}^{M-1}(x_{j,n}-x_{j+1,n}) |w_\alpha(x_{j,n})P(x_{j,n})|^p\right|\le \eta/2.
 $$
 Together with \eref{pf1eqn1}, this leads to \eref{eq:lpmzineq}.
 
 \end{proof}

 Recall the definition of $S_n$ from \eqref{eq:lsquare_def}. For the
 sake of the continuation of the proof, we (re-)define $S_n$ to
 hold for any
 $\{y_j\}_{j=1}^M\subset \RR$ (and, as before, for $\C=\{x_{M}<\cdots<x_{1}\}\subset \RR$).
\begin{equation}\label{eq:lsquare_def2}
S_n(\{y_j\}_{j=1}^M; \C)=\argmin_{P\in \Pi_n}\sum_{j=1}^{M-1} \left(y_j-P(x_{j})\right)^2 (x_{j}-x_{j+1})w_\alpha^2(x_{j}). 
\end{equation}

Clearly, the relationship between \eqref{eq:lsquare_def2} and  \eqref{eq:lsquare_def} is that
\begin{equation}
  \label{eq:relation-between-sn}
  S_n(g)\equiv S_n(\{w_{\alpha}^{-1}\left(x_j\right)g\left(x_j\right)\}_{j=1}^M).
\end{equation}
 Theorem~\ref{theo:ls_square_theo} will be deduced from the following proposition.
 \begin{prop}
 \label{prop:ls_square_prop}
 Let $\C_n$ be as in Theorem~\ref{theo:ls_square_theo}, and
 \eref{eq:special_meshnorm} be satisfied with $C(\alpha)=c(\alpha)/8$
 where $c(\alpha)$ is as in \prettyref{theo:mztheo}. Then for any $\{y_j\}_{j=1}^M\subset \RR$
 \begin{equation}\label{eq:ls_square_norm_bd}
 \|S_n(\{y_j\}_{j=1}^M; \C_n)w_\alpha\|_\infty \le c(x_{1,n}-x_{M,n})n^{1-1/\alpha}\max_{1\le j\le M}w_\alpha(x_{j,n})|y_j|.
 \end{equation}
 \end{prop}

 \begin{proof}\ %of Pospostion~\ref{prop:ls_square_prop}.
 Our assumptions imply that $\C_n=\{x_{j,n}\}_{j=1}^M$ satisfies the conditions of Theorem~\ref{theo:mztheo} with $p=2$, $\eta=1/4$. Let $\nu_n$ denote the measure that associates the mass $(x_{j,n}-x_{j+1,n})w_\alpha^2(x_{j,n})$ with $x_{j,n}$, $1\le j\le M-1$. 
 Let $\mathbf{G}$ be the matrix defined by
 \begin{equation}\label{eq:grammatrixdef}
 \mathbf{G}_{j,k} = \int_\RR p_k(t)p_j(t)d\nu_n(t), \qquad j,k=0,1,\cdots,n.
 \end{equation}
 If $P=\sum_{k=0}^n a_kp_k\in \Pi_n$, then $\|w_\alpha P\|_2^2=\sum_{k=0}^n a_k^2$, and
 $$
 \sum_{j=1}^{M-1}(x_{j,n}-x_{j+1,n}) |w_\alpha(x_{j,n})P(x_{j,n})|^2= \sum_{j,k=0}^n \mathbf{G}_{j,k}a_ja_k.
 $$
 Therefore, \eref{eq:lpmzineq} (used with $p=2$, $\eta=1/4$) can be rewritten in the form
 \begin{equation}\label{eq:singularvalues}
 (3/4)|\mathbf{a}|_2^2 \le \mathbf{a}^T\mathbf{G}\mathbf{a} \le (5/4)|\mathbf{a}|_2^2, \qquad \mathbf{a}\in\RR^{n+1}.
 \end{equation}
 This implies that $\mathbf{G}$ is positive definite, and hence, invertible. Since $\mathbf{G}$ is symmetric, so is $\mathbf{G}^{-1}$ symmetric and positive definite. Let $\mathbf{R}$ be a right triangular matrix so that the Cholesky decomposition $\mathbf{G}^{-1}=\mathbf{R}\mathbf{R}^T$ holds. The estimates \eref{eq:singularvalues} implies that
 \begin{equation}\label{eq:gram_norm_ests}
 (4/5)|\mathbf{a}|_2^2 \le |\mathbf{R}\mathbf{a}|_2^2 \le (4/3)|\mathbf{a}|_2^2, \qquad \mathbf{a}\in\RR^{n+1}.
 \end{equation}
 It is easy to see from the definitions that the system of polynomials defined by
 \begin{equation}\label{eq:disc_orth_def}
 \tilde{p}_k(t)=\sum_{j=0}^k \mathbf{R}_{j,k}p_k(t), \qquad k=0,\cdots, n,
 \end{equation}
 satisfies for $j,k=0,1,\cdots,n$,
 \begin{equation}\label{eq:disc_orthogonality}
 \int_\RR \tilde{p}_k(t)\tilde{p}_k(t)d\nu_n(t)=\left\{\begin{array}{ll}
 1, &\mbox{ if $k=j$,}\\
 0, &\mbox{otherwise.}
 \end{array}\right.
 \end{equation}
 Therefore, for $x\in\RR$,
 \begin{equation}\label{pf2eqn1}
 S_n(\{y_j\}_{j=1}^M; \C_n)(x)=\sum_{j=1}^{M-1}(x_{j,n}-x_{j+1,n})w_\alpha^2(x_{j,n}) y_j\sum_{\ell=0}^n \tilde{p}_\ell(x_{j,n})\tilde{p}_\ell(x).
 \end{equation}
 In particular, using Schwarz inequality, we get for all $x\in\RR$,
 \begin{eqnarray}\label{pf2eqn2}
 \left|w_\alpha(x)S_n(\{y_j\}_{j=1}^M; \C_n)(x)\right| &\le& \left\{\sum_{j=1}^{M-1}(x_{j,n}-x_{j+1,n})\right\}\left\{\max_{1\le j\le M-1}|w_\alpha(x_{j,n})y_j|\right\}\nonumber\\
 && \qquad \times\left\{\max_{1\le j\le M-1} w_\alpha(x)w_\alpha(x_{j,n})\left|\sum_{\ell=0}^n \tilde{p}_\ell(x_{j,n})\tilde{p}_\ell(x)\right|\right\}\nonumber\\
                                                       &\le& (x_{1,n}-x_{M,n})\left\{\max_{1\le j\le M-1}|w_\alpha(x_{j,n})y_j|\right\}\nonumber\\
   && \qquad \times \max_{t\in\RR}\left\{w_\alpha^2(t)\sum_{\ell=0}^n (\tilde{p}_\ell(t))^2\right\}.
 \end{eqnarray}
 Let $t\in\RR$, and $\mathbf{p}=(p_0(t),\cdots,p_n(t))^T\in\RR^{n+1}$. Using \eref{eq:disc_orth_def} and \eref{eq:gram_norm_ests}, we see that
 $$
 \sum_{\ell=0}^n (\tilde{p}_\ell(t))^2=|\mathbf{R}\mathbf{p}|_2^2\le (5/4)|\mathbf{p}|_2^2=(5/4)\sum_{\ell=0}^n p_\ell(t)^2.
 $$
 Thus, \eref{pf2eqn2} yields
 \begin{eqnarray}\label{pf2eqn3}
 \max_{x\in\RR}\left|w_\alpha(x)S_n(\{y_j\}_{j=1}^M; \C_n)(x)\right| &\le& 
(5/4) (x_{1,n}-x_{M,n})\nonumber\\
&& \qquad \times   \left\{\max_{1\le j\le M-1}|w_\alpha(x_{j,n})y_j|\right\}\nonumber\\
 && \qquad \times \max_{t\in\RR}\left\{w_\alpha^2(t)\sum_{\ell=0}^n p_\ell(t)^2\right\}.
 \end{eqnarray}
 It is proved in \cite[Theorem~3.2.5]{mhaskar1997introduction} that
 $$
 \max_{t\in\RR}\left\{w_\alpha^2(t)\sum_{\ell=0}^n p_\ell(t)^2\right\}\le cn^{1-1/\alpha}.
 $$
 Together with \eref{pf2eqn3}, this implies \eref{eq:ls_square_norm_bd}. 
 \end{proof}

% \noindent
% \textsc{Proof of Theorem~\ref{theo:ls_square_theo}.}\\
\begin{proof}[Proof of \prettyref{theo:ls_square_theo}]
  Let $P\in\Pi_n$ be arbitrary. It is clear from
  \eqref{eq:lsquare_def}, \eqref{eq:lsquare_def2},
  \eqref{eq:relation-between-sn} and linearity of $S_n$ that
  \begin{align}\label{eq:sn-p}
    \begin{split}
      S_n\bigl(\{g(x_{j,n})w_\alpha^{-1}(x_{j,n})-P(x_{j,n})\}_{j=1}^M\bigr)&=S_n\bigl(\{g(x_{j,n})w_\alpha^{-1}(x_{j,n})\}_{j=1}^M\bigr)-S_n\bigl(\{P(x_{j,n})\}_{j=1}^M\bigr)\\
      &=S_n(g)-P.
    \end{split}
  \end{align}
  Set $C(\alpha)$ as in \prettyref{prop:ls_square_prop}. Then
  \eqref{eq:ls_square_norm_bd} shows that
  \begin{eqnarray*}
    \|(f-S_n(g))w_\alpha\|_\infty &\le& \|(f-P)w_\alpha\|_\infty + \|(S_n(g)-P)w_\alpha\|_\infty \\
                                  &\leq& \|(f-P)w_\alpha\|_\infty + \biggl(c(x_{1,n}-x_{M,n})n^{1-1/\alpha}\\
                                  &&\qquad \times \max_{1\le j\le M}w_\alpha(x_{j,n})|g(x_{j,n})w_\alpha^{-1}(x_{j,n})-P(x_{j,n})|\biggr).
  \end{eqnarray*}
  Now, taking \eqref{eq:epsbd} into account,
  \begin{align*}
    \max_{1\le j\le M}w_\alpha(x_{j,n})\left|g(x_{j,n})w_\alpha^{-1}(x_{j,n})-P(x_{j,n})\right| &=
                                                                                                  \max_{1\le j\le M} w_\alpha(x_{j,n})\left|f(x_{j,n})-w_{\alpha}^{-1}(x_{j,n})\errfn(x_{j,n})-P(x_{j,n})\right|\\
                                                                                                &\leq \max_{1\le j\le M}|\errfn(x_{j,n})|+\|(f-P)w_{\alpha}\|_{\infty},
  \end{align*}
  and, since
  $\left[x_{M,n},x_{1,n}\right]\supseteq\Delta_n\left(2,\alpha,1/8\right)$,
  we conclude that
  \begin{align*}
    \|(f-S_n(g))w_\alpha\|_\infty &\leq \|(f-P)w_\alpha\|_\infty+c(x_{1,n}-x_{M,n})n^{1-1/\alpha} \left\{\|(f-P)w_\alpha\|_\infty + \ee\right\}\\
                                  &\le c(x_{1,n}-x_{M,n})n^{1-1/\alpha}\left\{\|(f-P)w_\alpha\|_\infty+\ee\right\}.
  \end{align*}

  Since $P\in\Pi_n$ was arbitrary, this completes the proof.
\end{proof}

 In order to extend the estimate in Theorem~\ref{theo:ls_square_theo} to the complex domain, it is tempting to use part (b) of Theorem~\ref{theo:mhas_saff_theo}.
 However, since $f-S_n(g)$ is not a polynomial, we cannot do so directly. 
 Therefore, we will estimate first $S_n(g)-L_n(f)$, and then use Theorem~\ref{theo:mhas_saff_theo}.

 %\noindent\textsc{Proof of Theorem~\ref{theo:maintheo}.}\\
\begin{proof}[Proof of \prettyref{theo:maintheo}]
 In view of \eref{eq:ls_square_approx_bd} and \eref{eq:uniformity}, we have 
 \begin{equation}\label{pf4eqn1}
 \|(f-S_n(g)) w_\alpha\|_\infty \precapprox \frac{\rho(\alpha,\tau,\lambda)^n}{n^{n/\lambda-n/\alpha}}+\ee.
 \end{equation}
 This proves \eref{eq:ls_square_approx_uniform}.

 Let $|z|\le r_n$.
 Using \eref{eq:uniformity2}, this implies that
 \begin{equation}\label{pf4eqn2}
 \|(L_n(f)-S_n(g)) w_\alpha\|_\infty \precapprox \frac{\rho(\alpha,\tau,\lambda)^n}{n^{n/\lambda-n/\alpha}}+\ee.
 \end{equation}
 Since $L_n(f)-S_n(g)\in \Pi_n$, we may use Theorem~\ref{theo:mhas_saff_theo} to obtain for $z\in\CC$:
 \begin{equation}\label{pf4eqn3}
 |L_n(f)(z)-S_n(g)(z)| \precapprox \left(\frac{\rho(\alpha,\tau,\lambda)^n}{n^{n/\lambda-n/\alpha}}+\ee\right)\exp(nU(z/a_n)-nF_\alpha).
 \end{equation}
 Together with \eref{eq:inter-error-mid} (applied with $b_n=r_n$),
 this leads to \eref{eq:ls_ext_mid}.

 Similarly, observing that for $|z|>r_n$,
 $$
 \exp(nU(z/a_n)) \precapprox \left(\frac{|z|}{a_n}\right)^n,
 $$
 and using \eref{pf4eqn3} we deduce that
 \begin{equation}\label{pf4eqn4}
 |L_n(f)(z)-S_n(g)(z)| \precapprox \exp\left(\tau|z|^\lambda\right)\left(1+\frac{n^{n/\lambda-n/\alpha}}{\rho(\alpha,\tau,\lambda)^n}\ee\right)\biggl[\left(\frac{|z|}{a_n}\right)^n e^{-nF_\alpha}\frac{\rho(\alpha,\tau,\lambda)^n}{n^{n/\lambda-n/\alpha}}\exp\left(-\tau|z|^\lambda\right)\biggr].
 \end{equation}
 Taking into account the definition of $\rho(\alpha,\tau,\lambda)$ and
 $a_n$ and $F_\alpha$, and the fact that\footnote{It can be easily
   verified that for $t\geq r_n$ the function $t^n \exp\left(-\tau
 t^\lambda\right)$ is decreasing in $t$.}
 $$
 |z|^n\exp(-\tau|z|^\lambda)\le \left(\frac{n}{\tau\lambda}\right)^{n/\lambda}\exp(-n/\lambda),
 $$
 we deduce that the expression in the square brackets in
 \eqref{pf4eqn4} is bounded from above by $1$, and we obtain
 \begin{align*}
 |L_n(f)(z)-S_n(g)(z)| \precapprox \exp(\tau|z|^\lambda)\left(1+\frac{n^{n/\lambda-n/\alpha}}{\rho(\alpha,\tau,\lambda)^n}\ee\right).
 \end{align*}
 Using \eref{eq:inter-error-large}, this leads to
 \eref{eq:ls_ext_large}. It is easy to verify that all the relations hold uniformly on compact
 subsets in $z$. 
 \end{proof}

\subsection{Asymptotics as $\ee\to 0$}
\label{subsec:asymptotics_eps}

 While \prettyref{theo:maintheo} does not require $\ee>0$, we
 examine in the rest of this section the noisy case when $\ee>0$ and
 prove \prettyref{thm:main-thm-exp-type}. In
 this case, the perturbation level $\ee$ would dominate the term
 $\disp \frac{\rho(\alpha,\tau,\lambda)^n}{n^{n/\lambda-n/\alpha}}$ if
 $n$ is too large.

 In this subsection we use the shorthand notation
 \begin{align}
   \mu &:= {1\over \lambda}-{1\over \alpha}\label{eq:mu-def}\\
   \rho&:=\rho\left(\alpha,\tau,\lambda\right)={\beta_{\alpha}\over 2} \left(\tau\lambda\right)^{1\over \lambda}\exp{\mu}\label{eq:rho-def},
 \end{align}
where $\beta_\alpha$ is defined in \eqref{eq:mrs-1st}.
 
 \begin{defn}\label{def:lambert}
   The Lambert's W-function $\lam$ is implicitly defined as the
   (multivalued) solution to the equation
   \begin{equation}
     \lam\left(ze^{z}\right)=z.\label{eq:lambert-identity}
   \end{equation}
 \end{defn}
 It is known that the single-valued branch $\lam>1$ satisfies \cite{hoorfar_inequalities_2008}
   \begin{equation}\label{eq:lambert-asympt}
     \lam\left(x\right)=\log x-\log\log x+o\left(1\right),\;x\to\infty
   \end{equation}
 
 \begin{defn}\label{def:optndef}
   Given $\alpha,\tau,\lambda$ and $\ee>0$, we define
   \begin{equation}
     \fn\left(\ee,\alpha,\tau,\lambda\right):=\left\lfloor \frac{1}{q}\log\frac{1}{\varepsilon}\right\rfloor ,\label{eq:mstar-def}
   \end{equation}
   where
   \begin{equation}
     q=q\left(\rho,\mu,\varepsilon\right)=\mu\lam\left(\frac{\rho^{-\frac{1}{\mu}}\log\frac{1}{\varepsilon}}{\mu}\right),\label{eq:q-def}
   \end{equation}
   $\mu,\rho$ are given by \eqref{eq:mu-def}, \eqref{eq:rho-def}, and
   $\lam$ is the Lambert's W-function (see \prettyref{def:lambert}).
 \end{defn}
                                                                                     
 We recall the asymptotic notations from
 \prettyref{def:asymptotics-eps} and \prettyref{def:asypmt-not-n}. The
 proposition below relates between these two, in our setting.
 \begin{prop}
   \label{prop:asymptotic-order-eps} Let $\fn$ be as defined in
   \eqref{eq:mstar-def}. Then:
   \begin{enumerate}

   \item $\fn$ is the largest integer $n$ such that
     \begin{equation}\label{eq:optndef}
       \rho^n n^{-\mu n} \ge \ee.
     \end{equation}

   \item For any sequence $A\left(n\right)$ with
     $A\left(n\right)\precapprox \rho^{n}n^{-\mu n}$
     (resp. $A\left(n\right)\succapprox \rho^{n}n^{-\mu n}$) it holds
     that $A\left(\fn\right)\lessapprox \ee$
     (resp. $A\left(\fn\right)\gtrapprox \ee$).
   \end{enumerate}

 \end{prop}

 \begin{proof}
    The exact solution to $\rho^{n}n^{-\mu n}=\ee$ is given by
    \[
      n=\frac{1}{q}\log\frac{1}{\ee},
    \]
    which can be checked by direct substitution.  In more detail:

    \begin{align*}
      \log\frac{1}{\ee} & =\log\left(\frac{n^{\mu}}{\rho}\right)^{n} 
                          = n\left[\log\frac{n^{\mu}}{\rho}\right] 
                          = n\mu\left[\log n-\frac{1}{\mu}\log\rho\right] = 
                          n\mu\log\left\{ n\rho^{-\frac{1}{\mu}}\right\}\\
      \frac{1}{\mu}\rho^{-\frac{1}{\mu}}\log\frac{1}{\ee} &= n\rho^{-\frac{1}{\mu}}\log\left\{ n\rho^{-\frac{1}{\mu}}\right\}.
    \end{align*}
    Applying $\lam$ to both sides and using \eqref{eq:lambert-identity}
    we have
    \[
      \lam\left(\frac{1}{\mu}\rho^{-\frac{1}{\mu}}\log\frac{1}{\ee}\right)=\log\left\{ n\rho^{-\frac{1}{\mu}}\right\} .
    \]
    Now since $\lam\left(x\right)\exp\lam\left(x\right)=x$, we have by
    exponentiation of the preceding formula
    \begin{align*}
      \frac{\frac{1}{\mu}\rho^{-\frac{1}{\mu}}\log\frac{1}{\ee}}{\lam\left(\frac{1}{\mu}\rho^{-\frac{1}{\mu}}\log\frac{1}{\ee}\right)} & =\exp\lam\left(\frac{1}{\mu}\rho^{-\frac{1}{\mu}}\log\frac{1}{\ee}\right)=n\rho^{-\frac{1}{\mu}}\\
      \frac{\log\frac{1}{\ee}}{\mu\lam\left(\frac{1}{\mu}\rho^{-\frac{1}{\mu}}\log\frac{1}{\ee}\right)} & =n.
    \end{align*}

    This proves \eqref{eq:optndef}.

    To show the second part of the proposition, put
    $a=a\left(\ee\right)=\fn\left(\ee\right)$ and
    $b=b\left(\ee\right)={1\over q}\log{1\over\ee}$, so that
    $b-1\leq a \leq b$ and $\rho^b b^{-\mu b}=\ee$. Denote
    $t:=1-{1\over b}$, so $a\geq bt$. We have

    \begin{equation*}
      \rho^a a^{-\mu a} \leq \rho^{bt} \left(bt\right)^{-\mu bt} 
      =\ee^t t^{-\mu bt}
      =\ee \left({1\over\ee}\right)^{{{q\over {\log{1\over\ee}}} + {\mu t \log {1\over t}\over q}}}.
    \end{equation*}

    Furthermore, for any $\delta>0$ we have
    \begin{equation*}
    \left(1+\delta\right)^a\leq
    (1+\delta)^b=\left(1\over\ee\right)^{\log(1+\delta)/q}.
  \end{equation*}
  
    Now, using \eqref{eq:lambert-asympt} and \eqref{eq:q-def}, clearly
    there exists $\ee_0$, depending on $\alpha,\tau,\lambda$ such that
    for $\ee<\ee_0$ we have
    \begin{equation}
      \label{eq:lambert-bounds-eps0}
      {\mu\over 2}\log\log{1\over\ee}\leq q(\ee)\leq 2\mu\log\log{1\over\ee}.
    \end{equation}

    Now let $c_1:=\eta>0$ be given as in
    \prettyref{def:asymptotics-eps}. Choosing
    $\delta=\exp\left({\eta\over 2}\right)-1$ (i.e.
    $\log\left(1+\delta\right)={\eta\over 2}$), we have from
    \prettyref{def:asypmt-not-n} that there exists
    $c>0$ such that
    \begin{align*}
      A(a) &\leq c\cdot (1+\delta)^a \rho^a a^{-\mu a}\\
           &\leq c\cdot  \ee \left({1\over\ee}\right)^{{q\over {\log{1\over\ee}}} + {\mu t \log {1\over t}\over q}+{\eta \over 2q}}
    \end{align*}

    Taking \eqref{eq:lambert-bounds-eps0} into account, it is
    sufficient to show that there exists
    $C_{\eta}$ such that 
    \begin{align*}
      \left({1\over\ee}\right)^{{q\over {\log{1\over\ee}}} + {\mu t \log {1\over t}\over q}+{\eta \over 2q}}
      \leq
      C_{\eta}\left({1\over\ee}\right)^{\eta\over q},\quad \ee<\ee_0.
    \end{align*}
    
    Taking logarithm of both sides, this is equivalent to
    \begin{align*}
      {\log{1\over\ee}\over q} \left({q^2\over{\log{1\over\ee}}}+\mu t \log{1\over t}-{\eta\over{2}}\right)\leq \log C_{\eta}.
    \end{align*}

    Clearly, the expression in the left-hand side attains a maximum
    for some $0<\ee_{\eta}<\ee_0$.  The ``$\lessapprox $" direction
    follows. The ``$\gtrapprox$'' direction is immediate since
    $\rho^a a^{-\mu a} \geq \rho^b b^{-\mu b} = \ee$.
  \end{proof}

 Denote
   \begin{equation}
     \expnn\left(z\right):=U\left(z\right)-F_\alpha= \int_{-1}^{1}\log\left|z-t\right|dv_{\alpha}\left(t\right)+{1\over\alpha}+\log 2,\label{eq:expn2-def}
   \end{equation}
   where $v_{\alpha}$ is the Ullman's distribution defined in
   \eqref{eq:ullmandist}. In case of $\alpha=2$, the function
   $\delta(z)$ is explicitly given by \eqref{eq:hermite_exponent}.
   
   Next, recall the definition of $r_{n}$ in \eqref{eq:rn-def}.
 \begin{prop}
   \label{prop:limits-of-expn} For any  $c\ge 1$ independent of $\ee$ we have
   \begin{numcases}{\lim_{\ee\to 0} \frac{\delta(z)}{q\left(\ee\right)}
       = }
     0 & if $|z|=c$ \label{eq:expn-lim-left}\\
     1 & for any $z=z_{\ee}$ with $|z_{\ee}|=c\cdot r_\fn/a_\fn$\label{eq:expn-lim-right}.
   \end{numcases}
 \end{prop}
 \begin{proof}
   Since $\lim_{\ee\to0}q=\infty$ and $\delta$ is independent of $\ee$,
   \eqref{eq:expn-lim-left} follows immediately.

   On the other hand,
   \[
     \frac{r_{n}}{a_{n}}=\left(\frac{n}{\tau\lambda}\right)^{\frac{1}{\lambda}}\beta_{\alpha}^{-1}n^{-\frac{1}{\alpha}}=\frac{e^{\mu}}{2\rho}n^{\mu},
   \]
   and therefore by a simple computation, using
   \eqref{eq:lambert-asympt} and \eqref{eq:q-def} we obtain
   \begin{equation*}
     \lim_{\ee\to 0}\frac{1}{q\left(\ee\right)}\log {cr_\fn\over a_\fn} = 1.
   \end{equation*}
   Therefore \eqref{eq:expn-lim-right} immediately follows:
   $$
   \lim_{\ee\to
     0}\frac{1}{q\left(\ee\right)}\left[\int_{-1}^1\log|z_\ee-t|dv_{\alpha}(t)-F_\alpha\right]=
   \lim_{\ee\to0}\frac{1}{q(\ee)}\left[
   \log\frac{cr_{\fn}}{a_{\fn}}+\underbrace{\int_{-1}^{1}\log\left|1-\frac{ta_{\fn}}{cr_{\fn}}\right|dv_{\alpha}\left(t\right)}_{\to0}-F_\alpha\right]=1.\qedhere
 $$
\end{proof}

 \begin{prop}
   \label{prop:asymptotic-order-exponent-fn}
   Let $A(n)$ and $B(n)$ be arbitrary sequences satisfying
   \begin{align*}
     A\left(n\right) & \precapprox \rho^n n^{-\mu n},\\
     B\left(n\right) &\succapprox \rho^n n^{-\mu n}.
   \end{align*}
   Then for $\ee\ll1$, and
   $1\lessapprox\left|z\right|\lessapprox r_{\fn}/a_\fn$ we have
   \begin{align*}
     A\left(\fn\right)\exp\bigl(\fn U\left(z\right)-\fn F_\alpha\bigr) & \lessapprox\ee^{1-{1\over q}\delta(z)},\\
     B\left(\fn\right)\exp\bigl(\fn U\left(z\right)-\fn F_\alpha\bigr) & \gtrapprox\ee^{1-{1\over q}\delta(z)}.
   \end{align*}
 \end{prop}
 \begin{proof}
   Plugging \eqref{eq:log-potential-def} and \eqref{eq:mstar-def}, and
   using an argument similar to \prettyref{prop:asymptotic-order-eps}, it
   is easy to see that, uniformly on compact sets in $z\in\CC$,
   \[
     \exp
     \bigl(\fn\left\{ U\left(z\right)-F_{\alpha}\right\}\bigr)
     \approx \ee^{-\frac{1}{q}\left[\int_{-1}^{1}\log\left|z-t\right|dv_{\alpha}\left(t\right)-F_\alpha\right]}.
   \]
   Combining this with \prettyref{prop:asymptotic-order-eps} and
   \eqref{eq:expn2-def} provides the conclusion.
 \end{proof}

 We are now in a position to complete the proof of
 \prettyref{thm:main-thm-exp-type}.

 \begin{proof}[Proof of \prettyref{thm:main-thm-exp-type}]
   We put $\fn=\fn\left(\ee,\alpha,\tau,\lambda\right)$ and
   \begin{equation}
     \label{eq:choice-of-constants}
     c_1:=C\left(\alpha\right)\qquad \textrm{ from \prettyref{theo:ls_square_theo}}
   \end{equation}

   Furthermore, by definition of $\Delta_n$ in \eqref{eq:Deltadef} for
   large enough $n$ we shall have
   $\Delta_n\left(2,\alpha,1/8\right)\subseteq\left[-{4\over
       3}a_n,{4\over 3}a_n\right]$. Therefore in particular for small
   enough $\ee$ \eqref{eq:sampling_extent_condition} and
   \eqref{eq:density_condition_main} will imply the sampling extent
   conditions of \prettyref{theo:maintheo}. Applying this theorem and
   rescaling $z\to a_\fn z$ we obtain, uniformly in compact subsets in $z$,
   \begin{numcases}{|f(a_\fn z)-S_\fn (g)( a_\fn z)| \precapprox}
     \ee\exp(\fn U(z)-\fn F_\alpha) & for $|z|\le r_\fn/a_\fn$,\label{eq:ls_ext_mid_eps}\\
     \exp\left(\tau|a_\fn z|^\lambda\right) & for $|z|>r_\fn/a_\fn$.
     \label{eq:ls_ext_large_eps}
   \end{numcases}

   Combining \eqref{eq:ls_ext_mid_eps}, \eqref{eq:ls_ext_large_eps},
   \eqref{eq:frostman} and \prettyref{prop:asymptotic-order-eps} we
   obtain \eqref{eq:main-extr-estimate-eps} \footnote{Indeed, if
     $|z|\le a_\fn,\;z\in\RR$ then
     $U\left(z/a_\fn\right)=|{\beta_\alpha z \over \beta_\alpha
       \fn^{1/\alpha}}|^\alpha+F_\alpha$ and so
     $\ee \exp\left(\fn U\left(z/a_{\fn} \right)-\fn
       F_\alpha\right)=\ee\exp\left(|z|^\alpha\right)$.}, clearly
   uniformly on compact subsets in $z$. The rest of the assertions
   follow from \prettyref{prop:asymptotic-order-exponent-fn} and
   \prettyref{prop:limits-of-expn}.
 \end{proof}

\subsection{Optimality}
We use the same notations as in
\prettyref{subsec:asymptotics_eps}.

\begin{proof}[Proof of \prettyref{thm:lower-bound-eps}]

  In view of \cite[Lemma~3.3]{canadapap}\footnote{ To reconcile the
    notation from \cite{canadapap}, note that the constant $F$ there
    is in fact equal to $-F_\alpha-\log \beta_\alpha$.}, we have for
  every polynomial $P\in\Pi_n$,
  \begin{equation}\label{pf5eqn1}
    \tn P\tn_{\tau,\lambda}\precapprox n^{\mu n}\rho^{-n}\|w_\alpha P\|_\infty.
  \end{equation}
  Therefore, there exists a sequence $M_n$ such that
  $\lim_{n\to\infty} M_n^{1/n}=1$, and the polynomials defined by
  \begin{equation}\label{pf5eqn2}
    P_n^*\left(z\right):=M_n^{-1}\rho^n n^{-\mu n}\frac{\mathcal{T}_{n}(z)}{\|\mathcal{T}_{n}w_{\alpha}\|_{\infty}}
  \end{equation}
  are all in $B_{\tau,\lambda}$ (just substitute
  $\|w_\alpha P_n^*\|_\infty=M_n^{-1}\rho^n n^{-\mu n}$ into
  \eqref{pf5eqn1}).  We also observe that
  $$
  \|w_\alpha P_\fn^*\|_\infty\le M_\fn^{-1} \ee \lessapprox \ee.
  $$

  Now, let $R$ (respectively, $I$) be any recovery (respectively, information) operator and \eref{eq:info_fn_cond} be satisfied. Then
  \begin{equation}\label{pf5eqn3}
    \|I(P_\fn^*)\|_Y\le M_\fn^{-1} \ee.
  \end{equation}
  Define $\xi\left(\ee\right):=M_\fn ^{-1}$. We have (cf. \eref{eq:worst_case_error})
  \begin{equation}\label{pf5eqn4}
    E(\tau,\lambda, \xi\left(\ee\right), z; R, I)\ge \bigl|P_\fn^*(a_\fn z)-R\left(I(P_\fn^*)-I(P_\fn^*)\right)(a_\fn z)\bigr| =\bigl|P_\fn^*(a_\fn z)-R(0)(a_\fn z)\bigr|,
  \end{equation}
  where $0$ denotes the zero element of $Y$. The same inequality holds also when $P_\fn^*$ is replaced by $-P_\fn^*$. Thus,
  \begin{equation}\label{pf5eqn5}
    |P_\fn^*(a_\fn z)| =\frac{1}{2}\left|\left(P_\fn^*(a_\fn z)-R(0)(a_\fn z)\right)-\left(-P_\fn^*(a_\fn z)-R(0)(a_\fn z)\right)\right|\le E(\tau,\lambda, \xi(\ee), z; R, I).
  \end{equation}
  In view of \prettyref{prop:chebyshev_prop} we have
  \begin{equation}
    \label{eq:rescaled-lower-bound-on-compact-sets}
    |P_\fn^*(a_\fn z)| \gtrapprox \ee\exp(\fn U(z)-\fn F_\alpha),
  \end{equation}
  the relation holding uniformly on compact sets for $z\in\CC\setminus\left[-1,1\right]$. Applying
  \prettyref{prop:asymptotic-order-exponent-fn} leads to
  \eqref{eq:optimal-lower-bound} and concludes the proof with the 
  ``dark object'' $f_{\ee}:=P_\fn^*$.
\end{proof}

 \bibliographystyle{abbrv}
 \bibliography{entire-extrapolation}

 \end{document}